\documentclass[10pt]{article}

\usepackage{amsfonts}
\usepackage{amsmath}
\usepackage{amssymb}
\usepackage{amsthm}
\usepackage{mathrsfs}
\usepackage{color}
\usepackage{hyperref}
\usepackage{a4}
\usepackage{fullpage}
\usepackage{url}
\usepackage{hyperref}
\usepackage{graphicx}
\usepackage[nosort]{cite}
\usepackage{enumitem}
\providecommand{\tabularnewline}{\\}
\usepackage{multirow}
\usepackage[title]{appendix}

\newtheorem{theorem}{Theorem}

\newtheorem{example}[theorem]{Example}

\begin{document}
	
	\title{\textbf{Minimal System of Generators and Syzygies of Centro-Affine Invariants and Covariants for Homogeneous Planar Cubic Differential Systems with Free Terms and Linear Part}}
	\author{\textbf{Anis Hezzam}\thanks{%
			Faculty of Mathematics, University of Sciences and Technology Houari
			Boumediene, BP 32 El Alia 16111 Bab Ezzouar Algiers Algeria, E-mail : %
			\url{ahezzam@usthb.dz}} , \textbf{Dahira Dali}\thanks{%
			Faculty of Mathematics, University of Sciences and Technology Houari
			Boumediene, BP 32 El Alia 16111 Bab Ezzouar Algiers Algeria, E-mail : %
			\url{dddahira@gmail.com}}}
		
	\date{}
	\date{\today }
	\maketitle
	
	\begin{abstract}
A minimal system of generators of the algebra of the centro-affine covariants for homogeneous planar cubic differential systems with linear part is known. With the help of the Gurevich theorem avoiding the Aronhold's identities based on the calculation of determinants, we describe the algebra of the centro-affne covariants for homogeneous planar cubic differential systems with free terms.  
	\end{abstract}
	
	\bigskip
	
	\textbf{2010 MR Subject Classification:} 34C20, 15A72, 13A50.

\section{Introduction and Preliminaries}

Using Einstein's notation the polynomial differential systems in $n$ unknown
variables and of finite degree with coefficients in a field $\Bbbk $ ($\Bbbk
=%
\mathbb{R}
$ or $
\mathbb{C}
$) can be written as 
\begin{equation}
\frac{dx^{j}}{dt}=\sum_{r\in \Omega }a_{\alpha _{1}\cdots \alpha
	_{r}}^{j}x^{\alpha _{1}}\cdots x^{\alpha _{r}},  \label{system}
\end{equation}%
where $\Omega $ is a finite set of distinct natural numbers and for $%
j=1,...,n$ and $r\in \Omega $ $,$ $a_{\alpha _{1}\cdots \alpha
	_{r}}^{j}x^{\alpha _{1}}\cdots x^{\alpha _{r}}$.

The qualitative investigation of polynomial differential systems by means of
the algebraic invariants is developed by Konstantin Sergeevich Sibirskii
\cite{Dana Schlomiuk} when he assimilated the coefficients of
the polynomial differential systems (\ref{system}) to tensor components,
then made group action on the phase space of these systems and classified
geometric properties of planar quadratic differential systems with the help
of algebraic and semi-algebraic relations in terms of their coefficients.
Many difficult results about differential systems have been obtained with
the help of invariant theory, we refer the readers to  \cite{Sibirskii, Sibirskii1, Sibirskii2}.

The theory of invariants motivated by projective geometry, number theory and algebraic geometry since Johann Carl Friedrich Gauss has published his Latin book "Disquisitiones arithmeticae" in 1801 on the representations of integers observed an invariant behavior in the theory of quadratic forms under the action of $SL(2,\mathbb{C})$ by taking $x$ to $Ax+By$ and $y$ to $Cx+Dy$. Arthur Cayley and James Joseph Sylvester are the first mathematicians who consider the invariants of algebraic forms but they never succeed in calculating it. Many mathematicians are devoted to the invariants of algebraic among them Georges Salmon, Charles Hermite, George Boole, Siegfried Heinrich Aronhold, Rudolf Friedrich Alfred Clebsch and Paul Albert Gordan who succeed in demonstrating the existence of a generating family of the invariants of binary forms of finite degree by using the symbolic calculus invented by Gordan. Fa\`{a} Di Bruno worked on invariants and wrote a book ”Théorie des formes binaires” wich was regarded highly by David Hilbert who solved the fineteness problem of the invariant theory by using Noetherianity in 1890. In 1993, Hilbert had solved the major problems in the invariant theory \cite{DahiraDaliandSuiSunCheng, Gurevich}. 



When we use the invariant theory, the main problem that we are facing
is the knowledge of the invariants, that is \cite{Popov}. In the case where the
action group is the linear general group $GL(n,\Bbbk )$, the $\Bbbk $%
-algebra of the algebraic invariants of systems (\ref{system}) called
centro-affine covariants is of finite type. However, the description of the
algebra is a difficult matter. One has to manipulate polynomial in several
variables. The invariants for bivariate quadratic differential systems are
polynomials in 12 indeterminates. Minimal system of generators of
centro-affine invariants for differential systems where $\Omega =\left \{
0,1,2\right \} $ in \cite{Sibirskii,Vulpe,BoularasandDali} and $%
\Omega =\left \{ 1,3\right \} $ are known in \cite{VMChebanu, Vulpecubic} and
minimal system of $27$ syzygies relating elements of the minimal system of (%
\ref{system}) where $\Omega =\left \{ 0,1,2\right \} $ is found in \cite%
{DanilyukandSibirskii}.

In this paper, using the theorem of Gurevich \cite{Gurevich}, we develop a
constructive to describe the $\Bbbk -$algebra of the centro-affine
covariants for a given family of polynomial differential systems avoiding
Aronhold's identities. This method can be used to determinate
syzygies between the elements of a given minimal system of generators of
centro-affine covariants. Then, find a minimal system of generators of the
centro-affine covariants for the planar cubic differential systems, that is,
systems (\ref{system}) where $\Omega =\left \{ 0,1,3\right \} $ and a
corresponding minimal generator system of the syzygies for these
differential systems. The syzygies between the relating elements of the
minimal system of a given minimal system of generators of centro-affine
covariants for a given polynomial differential systems allows us to have a
unique decomposition of centro-affine covariants of the considered
differential systems.

We denote by $\mathcal{C}(n,\Bbbk ,\Omega )$ the set of all coefficients on
the right hand side of polynomial differential systems (\ref{system}).
Denote by $\mathcal{S}(n,\Bbbk ,\Omega )$ and can be identified as a direct
sum $\bigoplus {}_{r\in \Omega }\mathcal{T}_{r}^{1},$ where $\mathcal{T}%
_{r}^{1}$ denotes the $\Bbbk $-vectorial space of tensors $1$ time
contravariant and $r$ times covariant for these systems (\ref{system}).

The action of $GL(n,\Bbbk )$ on the plane $(q,x)\mapsto qx$ induces a
representation $\rho $ on $GL(\mathcal{C}(n,\Bbbk ,\Omega ))$ defined for
all $r\in \Omega $ by

\begin{equation}
\rho (q)a_{\alpha _{1}\alpha _{2}...\alpha _{r}}^{j}=q_{i}^{j}p_{\alpha
	_{1}}^{\beta _{1}}...p_{\alpha _{r}}^{\beta _{r}}a_{\beta _{1}\beta
	_{2}...\beta r}^{i} , \label{TransfLows}
\end{equation}
where $p$ is the inverse of the matrix $q$ called lows of centro-affine
transformations.\\
 A polynomial function $C(a,x):\mathcal{C}(n,\Bbbk ,\Omega
)\times \Bbbk ^{n}\rightarrow $ $\Bbbk $ is said to be a covariant with
respect to $GL(n,\Bbbk )$ or a centro-affine covariant of $\mathcal{C}%
(n,\Bbbk ,\Omega )\times \Bbbk ^{n}$ or simply a centro-affine covariant for 
$\mathcal{S}(n,\Bbbk ,\Omega )$ if

 $$\forall q\in GL(n,\Bbbk ),\forall a\in 
\mathcal{C}(n,\Bbbk ,\Omega ),\; C(\rho (q)a,qx)=(\det q)^{-w}C(a,x),$$  where $(\det
q)^{-w}$ is the character of $GL(n,\Bbbk )$ and $w\in 
\mathbb{Z}
$ is called the weight of $C(a,x).$ If $C(a,x)$ is constant with respect to $%
x$, it is called centro-affine invariant and written $C(a)$.

If $w \equiv 1$,\ the covariant is said to be
absolute, otherwise it is said to be relative.

A G-covariant $C(a,x)$ is said to be reducible if it can expressed as
polynomial function of G-covariants of lower degree, we write $C(a,x)\equiv 0$
(modulo G-covariants of lower degree). A finite family $\mathcal{B}$ of
centro-affine covariants for $\mathcal{S}(n,\Bbbk ,\Omega )$ is called a
system of generators if any centro-affine covariant for $\mathcal{S}(n,\Bbbk ,\Omega )$ is reducible to zero modulo $\mathcal{B}$. A system $\mathcal{B}
$ of generators is said to be minimal if none of them is generated by the
others.

\begin{example}
	\label{expcovariant}
	
	We can check with the help of the lows of centro-affine transformations (\ref%
	{TransfLows}) that $C(a)=a_{\alpha }^{\alpha }$ is a centro-adffine
	invariant for differential systems $\mathcal{S}(n,\Bbbk
	,\{0,1,2,...,k\})$ of weight $w=0$, $k\in 
	\mathbb{N}
	$. Indeed, for all $q\in GL(2,%
	\mathbb{R}
	)$, $C(\rho (q)a)=\rho (q)a_{\alpha }^{\alpha }=q_{i}^{\alpha }p_{\alpha
	}^{j}a_{j}^{i}=$ $\delta _{i}^{j}a_{j}^{i}=a_{i}^{i}=C(a)$, where $\alpha
	,i,j=1,...,n$. Some other examples are given in the following table:%
	

	\begin{center}
	
	\begin{tabular}{|c|c|c|}
		\hline 
		$\mathcal{C}(n,\Bbbk,\Omega)$ & Covariant & Weight\tabularnewline			\hline 
		\hline 
		\multirow{3}{*}{$\mathcal{C}(2,\mathbb{{R}},\{1,2\})$} & $a_{\alpha}^{\alpha},\alpha=1,2$ & $0$\tabularnewline
		\cline{2-3} 
		& $a_{p}^{r}a_{q}^{s}\varepsilon^{pq}\varepsilon_{rs},p,q,r,s=1,2$ & $0$\tabularnewline
		\cline{2-3} 
		& $a_{\gamma}^{\alpha}a_{\alpha p}^{\beta}a_{\beta q}^{\gamma}\varepsilon^{pq},\alpha,\beta,\gamma,p,q=1,2$ & $1$\tabularnewline
		\hline 
		\multirow{2}{*}{$\mathcal{C}(3,\mathbb{{R}},\{1,2\})$} & $a_{\alpha}^{\alpha},\alpha=1,3$ & $0$\tabularnewline
		\cline{2-3} 
		& $a_{\gamma p}^{\alpha}a_{\alpha q}^{\beta}a_{\beta s}^{\gamma}\varepsilon^{pqs},\alpha,\beta,\gamma,p,q,s=1,3$ & $1$\tabularnewline
		\hline 
		\multirow{2}{*}{$\mathcal{C}(2,\mathbb{{R}},\{0,1,3\})$} & $a^{\alpha}a^{\beta}a_{\nu}^{\gamma}a_{\mu\delta\gamma}^{\delta}a_{\beta pr}^{\mu}a_{\alpha qs}^{\nu}\varepsilon^{pq}\varepsilon^{rs}$, & \multirow{2}{*}{$2$}\tabularnewline
		& $\alpha,\beta,\gamma,\delta,\mu,\nu,p,q,r,s=1,2$ & \tabularnewline
		\hline 
		$\mathcal{C}(2,\mathbb{{R}},\{0,3\})$ & $a^{\alpha}a^{\beta}a^{\gamma}a^{p}a_{\alpha\beta\gamma}^{q}\varepsilon_{pq},\alpha,\beta,\gamma,p,q=1,2$ & $-1$\tabularnewline
		\hline 
	\end{tabular}
\end{center}



\end{example}

Let $E$ be a $\Bbbk $-vector space of dimension $n$, and let $p$ and $q$ be
two non negative integers.

A contraction over the tensor space $E^{\otimes p}\otimes E^{\ast \otimes q}$
is the map:%
\begin{equation*}
\varphi :E^{\otimes p}\otimes E^{\ast \otimes q}\rightarrow E^{\otimes
	p-1}\otimes E^{\ast \otimes q-1},
\end{equation*}%
defined by%
\begin{equation*}
\varphi (\xi _{j_{1}...j_{q}}^{i_{1}...i_{p}})=\sum_{\alpha =1}^{n}\xi
_{j_{1}...j_{l-1\alpha }j_{l+1}...j_{q}}^{i_{1}...i_{m-1}\alpha
	i_{m+1}...i_{p}}.
\end{equation*}

If $p=q$ then a sequence of $p$ contractions is called a complete
contraction.

A covariant alternation over the tensor space $E^{\otimes p}\otimes E^{\ast
	\otimes q}$ with $p,q\geq n$ is the map: 
\begin{equation*}
\phi :E^{\otimes p}\otimes E^{\ast \otimes q}\rightarrow E^{\otimes
	p}\otimes E^{\ast \otimes q-n},
\end{equation*}%
defined by%
\begin{equation*}
\phi (\xi
_{j_{1}...j_{q}}^{i_{1}...i_{p}})=\sum_{k_{1}=1}^{n}\sum_{k_{2}=1}^{n}...%
\sum_{k_{n}=1}^{n}\xi _{j_{1}...\alpha _{_{k_{1}}...}\alpha
	_{k_{2}}...\alpha _{k_{n}}...j_{q}}^{i_{1}...i_{p}}\varepsilon ^{\alpha
	_{k_{1}}\alpha _{k_{2}}...\alpha _{k_{n}}},
\end{equation*}%
and a contravariant alternation over the tensor space $E^{\otimes p}\otimes
E^{\ast \otimes q}$ with $p,q\geq n$ is the map:%
\begin{equation*}
\psi :E^{\otimes p}\otimes E^{\ast \otimes q}\rightarrow E^{\otimes
	p-n}\otimes E^{\ast \otimes q},
\end{equation*}%
defined by%
\begin{equation*}
\psi (\xi
_{j_{1}...j_{q}}^{i_{1}...i_{p}})=\sum_{k_{1}=1}^{n}\sum_{k_{2}=1}^{n}...%
\sum_{k_{n}=1}^{n}\xi _{j_{1}...j_{q}}^{i_{1}...\alpha _{_{k_{1}}...}\alpha
	_{k_{2}}...\alpha _{k_{n}}...i_{p}}\varepsilon _{\alpha _{k_{1}}\alpha
	_{k_{2}}...\alpha _{k_{n}}},
\end{equation*}%
where the tensor $\varepsilon _{\alpha _{k_{1}}\alpha _{k_{2}}...\alpha
	_{k_{n}}}(\varepsilon ^{\alpha _{k_{1}}\alpha _{k_{2}}...\alpha _{k_{n}}})$
with $\alpha _{k_{1}},\alpha _{k_{2}}...,\alpha _{k_{n}}=1,2,...,n$ is an $%
n $-vector, the valence of which coincides with the dimension of the space
and the coordinates of which are equal to%

$$
\varepsilon _{\alpha _{k_{1}}...\alpha _{k_{n}}}=\varepsilon ^{\alpha _{k_{1}}...\alpha _{k_{n}}}= \left \{ 
\begin{array}{cc}
\;\;1 & \text{if }(\alpha _{k_{1}},...,\alpha _{k_{n}})\text{ is an even
	permutation} \\ 
-1 & \text{if }(\alpha _{k_{1}},...,\alpha _{k_{n}})\text{ is an odd
	permutation} \\ 
\;\;0 & \text{ otherwise}%
\end{array}%
\right. .
$$

The $\Bbbk $-algebra $\Bbbk \lbrack \mathcal{C}(n,\Bbbk ,\Omega )\times
\Bbbk ^{n}]^{GL(n,\Bbbk )}$ is of finite type \cite{Hilbert} and in view of
the fundamental theorem of Gurevich \cite{Gurevich} it is generated by
polynomial expressions obtained by applying successive alternations and
complete contraction over the tensor product 
\begin{equation}
\left( \mathcal{T}_{r_{1}}^{1}\right) ^{\ast \otimes d_{r_{1}}}\otimes
...\otimes \left( \mathcal{T}_{r_{i}}^{1}\right) ^{\ast \otimes
	d_{r_{i}}}\otimes \left( \Bbbk ^{n}\right) ^{\ast \otimes \delta },
\label{tensProd}
\end{equation}%
where $r_{1},...,r_{i}\in \Omega ,$ $p=d_{r_{1}}+...+d_{r_{i}}+\delta $, $%
q=(r_{1}-1)d_{r_{1}}+...+(r_{i}-1)d_{r_{i}}$ and 
\begin{equation}
d_{r_{1}}+....+d_{r_{i}}-d_{0}-\delta \equiv 0[n] \label{ConInv}
\end{equation}

This motivate the following definition. An homogeneous centro-affine
covariant $C(a,x)$ for systems $\mathcal{S}(n,\Bbbk ,\Omega )$ is said to be
of type $(d_{r_{1}},...,d_{r_{i}},\delta )$ if and only if it is homogeneous of degree $d_{r_{i}}$ in relation to $%
a_{\alpha _{r_{1}}...\alpha _{r_{i}}}^{j},r_{i}\in \Omega $ and of degree $%
\delta $ in relation to the contravariant vector $x$, and satisfying the relation(\ref%
{ConInv}), here $r_{1},...,r_{i}\in \Omega
,\;d_{r_{1}},...,d_{r_{i}},\delta \in 
\mathbb{N}
$. A centro-affine covariant of type $(d_{r_{1}},...,d_{r_{i}},\delta )$ is said to be of degree $%
d=\delta +d_{r_{1}}+...+d_{r_{i}}$.

Hence, the generators of centro-affines covariants for $\mathcal{S}%
(n,\Bbbk ,\Omega )$ of given type $(d_{r_{1}},...,d_{r_{i}},\delta )$ can
be obtained with the help of successive alternations and complete
contraction from the tensor $t_{\beta _{1}...\beta _{q}}^{\alpha
	_{1}...\alpha _{p}}$, $p$ times contravariants and $q$ times covariants
belonging in the tensor product (\ref{tensProd}) where $%
p=d_{r_{1}}+...+d_{r_{i}}+\delta $ and $q=d_{r_{1}}+...+(r_{i}-1)d_{r_{i}}$
defined by

\begin{equation}
a_{\alpha _{1}...\alpha _{r_{1}}}^{i_{1}}\otimes ...\otimes a_{\beta
	_{1}...\beta _{r_{1}}}^{i_{d_{r_{1}}}}\otimes ...\otimes a_{\gamma
	_{1}...\gamma _{r_{i}}}^{_{j_{_{1}}}}\otimes ...\otimes a_{\delta
	_{1}...\delta _{r_{i}}}^{j_{d_{r_{i}}}}\otimes x^{m_{1}}\otimes ...\otimes
x^{m_{\delta }},  \label{tensor}
\end{equation}%
where $r_{1},...,r_{i}\in \Omega ,d_{r_{1}},...,d_{r_{i}},\delta \in 
\mathbb{N}
$ and the indices are belonging in $\left \{ 1,...,n\right \} $ and can be
written

\begin{equation}
\underset{\text{ }d_{r_{1}}\text{times}}{\underbrace{a_{\underset{\text{ }%
				r_{1}\text{times}}{\underbrace{\__{...}\_}}}^{-}...a_{\underset{\text{ }r_{1}%
				\text{times}}{\underbrace{\__{...}\_}}}^{-}}}...\underset{\text{ }d_{r_{i}}%
	\text{times}}{\underbrace{a_{\underset{\text{ }r_{i}\text{times}}{%
				\underbrace{\__{...}\_}}}^{-}...a_{\underset{\text{ }r_{i}\text{times}}{%
				\underbrace{\__{...}\_}}}^{-}}}.\label{tonsorfurmula}
\end{equation}

For instance, the generators of centro-affine covariants for homogeneous
planar quadratic differential systems $\mathcal{S}(2,%
\mathbb{R}
,\left \{ 2\right \} )$ of type $(d_{2},\delta )$ can be obtained with the
help of successive alternations and complete contraction from the tensor $%
(d_{2}+\delta )$ times contravariants and ($2d_{2})$ times covariants where $%
d_{2}-\delta \equiv 0[3]$ defined by

\begin{equation*}
a_{\alpha _{1}\alpha _{2}}^{l_{1}}\otimes ...\otimes a_{\beta _{1}\beta
	_{2}}^{l_{d_{2}}}\otimes x^{m_{1}}\otimes ...\otimes x^{m_{\delta }},
\end{equation*}%

where the indices belonging in $\left \{ 1,2\right \} $ and can be written

\begin{equation*}
\underset{\text{ }d_{2}\text{times}}{\underbrace{a_{\_ \text{ }%
			\_}^{-}...a_{\_ \text{ }\_}^{-}}}\underset{\text{ }\delta \text{times}}{%
	\underbrace{x^{-}...x^{-}}}.
\end{equation*}%
The generators of type $(1,1)$ can be obtained from the tensor $a_{\_ \text{ 
	}\_}^{-}x^{-}$ with the help of successive alternations and complete
contraction:

\begin{equation*}
\begin{array}{ccc}
a_{\_ \text{ }\_}^{-}x^{-} & \longmapsto & a_{\alpha \beta }^{\alpha }x.
\end{array}%
\end{equation*}%
The generators of type $(2,2)$ can be obtained from the tensor $a_{\_ \text{ 
	}\_}^{-}a_{\_ \text{ }\_}^{-}x^{-}x^{-}$ with the help of successive
alternations and complete contraction:

\begin{equation*}
\begin{array}{ccc}
a_{\_ \text{ }\_}^{-}a_{\_ \text{ }\_}^{-}x^{-}x^{-} & \longmapsto & \left
\{ 
\begin{array}{c}
a_{\alpha \gamma }^{\alpha }a_{\beta \delta }^{\beta }x^{\gamma }x^{\delta }
\\ 
\\ 
a_{\beta \gamma }^{\alpha }a_{\alpha \delta }^{\beta }x^{\gamma }x^{\delta }
\\ 
\\ 
a_{\alpha \beta }^{\alpha }a_{\gamma \delta }^{\beta }x^{\gamma }x^{\delta }%
\end{array}%
\right.%
\end{array}%
\end{equation*}%
The generators of type $(2,4)$ can be obtained from the tensor $a_{\_ \text{ }%
	\_}^{-}a_{\_ \text{ }\_}^{-}x^{-}x^{-}x^{-}x$ with the help of successive
alternations and complete contraction:

\[
\begin{array}{ccc}
a_{\_\text{ }\_}^{-}a_{\_\text{ }\_}^{-}x^{-}x^{-}x^{-}x^{-} & \longmapsto & \left\{ \begin{array}{c}
a_{\alpha\gamma}^{p}a_{\beta\delta}^{q}x^{\alpha}x^{\beta}x^{\gamma}x^{\delta}\varepsilon_{pq}\\
\\

a_{\alpha\gamma}^{p}a_{\beta\delta}^{\alpha}x^{q}x^{\beta}x^{\gamma}x^{\delta}\varepsilon_{pq}\\
\\

a_{\beta\gamma}^{p}a_{\alpha\delta}^{\alpha}x^{q}x^{\beta}x^{\gamma}x^{\delta}\varepsilon_{pq}
\end{array}\right.\end{array}
\]

where $\alpha ,\beta ,\gamma ,\delta ,p,q=1,2$ and $\varepsilon _{pq}=\sigma
(p,q)=p-q.$

The generators of centro-affine covariants for homogeneous planar quadratic
differential systems $\mathcal{S}(3,%
\mathbb{R}
,\left \{ 2\right \} )$ of type $(d_{2},\delta )$ can be obtained with the
help of successive alternations and complete contraction from the tensor $%
(d_{2}+\delta )$ times contravariants and ($2d_{2})$ times covariants where $%
d_{2}-\delta \equiv 0[3]$ defined by

\begin{equation*}
a_{\alpha _{1}\alpha _{2}}^{l_{1}}\otimes ...a_{\alpha _{1}\alpha
	_{2}}^{l_{1}}\otimes x^{m_{1}}\otimes ...\otimes x^{m_{\delta }},
\end{equation*}%
where the indices belonging in $\left \{ 1,2,3\right \} $ and can be written

\begin{equation*}
\underset{\text{ }d_{2}\text{times}}{\underbrace{a_{\_ \text{ }%
			\_}^{-}...a_{\_ \text{ }\_}^{-}}}\underset{\text{ }\delta \text{times}}{%
	\underbrace{x^{-}...x^{-}}}.
\end{equation*}%
The generators of type $(1,1)$ can be obtained from the tensor $a_{\--}^{-}x^{-}$ with the help of successive alternations and complete
contraction:

\[
\begin{array}{ccc}
a_{\_\text{ }\_}^{-}x^{-} & \longmapsto &a_{\alpha\beta}^{\alpha}x^{\beta},\end{array}\; a_{\beta\beta}^{\alpha}x^{\alpha},\; a_{q\alpha}^{p}x^{\alpha}\varepsilon^{pq},\; a_{\alpha\alpha}^{p}x^{q}\varepsilon^{pq},\; a_{\alpha p}^{\alpha}x^{q}\varepsilon^{pq}
\]

The generators of type $(2,0)$ can be obtained from the tensor  
 $a_{--}^{-}a_{--}^{-}$ with the help of successive alternations
and complete contraction\cite[page 88]{Sibirskii}:

\[
\begin{array}{lcl}
a_{\_\text{ }\_}^{-}a_{\_\text{ }\_}^{-} & \longmapsto & \left\{ \begin{array}{cc}
a_{\alpha\beta}^{\alpha}a_{\gamma\gamma}^{\beta}, &   a_{\beta\gamma}^{\alpha}a_{\beta\gamma}^{\alpha}\\
\\
a_{\alpha\gamma}^{\alpha}a_{\beta\gamma}^{\beta}, &   a_{\beta\gamma}^{\alpha}a_{\gamma\alpha}^{\beta}\\
\\
a_{\beta\beta}^{\alpha}a_{\gamma\gamma}^{\alpha},
\end{array}\right.\\
\\

a_{\_\text{ }\_}^{-}a_{\_\text{ }\_}^{-}\varepsilon^{--} & \longmapsto & \left\{ \begin{array}{c}
a_{\alpha p}^{\alpha}a_{\beta\beta}^{q}\varepsilon^{pq}\\
\\

a_{\beta p}^{\alpha}a_{\beta\alpha}^{q}\varepsilon^{pq}
\end{array}\right.
\end{array}
\]
where $\alpha ,\beta ,\gamma ,\delta ,p,q=1,2,3$ and $\varepsilon
_{pqr}=\sigma (p,q,r)$.

\section{Constructive method to describe the Algebra of Centro-Affine
	Covariants for $\mathcal{S}(n,\Bbbk ,\Omega )$}

The $\Bbbk $-algebra $\Bbbk \lbrack \mathcal{C}(n,\Bbbk ,\Omega )\times
\Bbbk ^{n}]^{GL(n,\Bbbk )}$ is multigraded and can be written

\begin{equation*}
\Bbbk \lbrack \mathcal{C}(n,\Bbbk ,\Omega )\times \Bbbk ^{n}]^{GL(n,\Bbbk
	)}=\bigoplus {}_{d_{r_{1}},...,d_{r_{i}},\delta \in \mathbb{N}}\mathcal{A}_{(d_{r_{1}},...,d_{r_{i_{0}}},\delta )},
\end{equation*}%
where  $\mathcal{A}_{(d_{r_{1}},...,d_{r_{i}},\delta )}$ denotes the
vectorial subspace of the homogeneous centro-affie covariants of type $%
(d_{r_{1}},...,d_{r_{i}},\delta ),\; r_{1},...,r_{i}\in \Omega
,\; d_{r_{1}},...,d_{r_{i}},\delta \in 
\mathbb{N}
$.

Indeed, the centro-affine covariant $C(a,x)$ for $\mathcal{S}(n,\Bbbk ,\Omega )$
can be decomposed as finite sum 
\begin{equation*}
C(a,x)=\sum_{i=1,...,s}C_{i}(a,x),
\end{equation*}%
where for $i=1,...,s$, $C_{i}(a,x)$ is a homogeneous polynomial in $\Bbbk
\lbrack \mathcal{C}(n,\Bbbk ,\Omega )\times \Bbbk ^{n}]$ of degree $d_{i}$
and for \ $i=1,...,s,$ $C_{i}(a,x)$ is a centro-affine covariant since for
all $q\in GL(n,\Bbbk )$ one has

\begin{equation*}
C(\rho (q)a,qx)=(\det q)^{-w}C(a,x),
\end{equation*}%
where $w$ be the weight of the centro-affine covariant $C(a,x)$. Therefore, 
\begin{equation*}
C(\rho (q)a,qx)=(\det q)^{-w}\sum_{i=1,...,s}C_{i}(a,x)
\end{equation*}%
or
\begin{equation*}
C(\rho (q)a,qx)=\sum_{i=1,...,s}(\det q)^{-w}C_{i}(a,x)
\end{equation*}%
and for all scalar $\lambda \in $ $\Bbbk ^{\ast }$, $C_{i}(a,x)=\lambda
^{-d_{i}}C_{i}(\lambda a,\lambda x)$, $i=1,...,s$ . Then

\begin{equation*}
\sum_{i=1,...,s}\lambda ^{d_{i}}C_{i}(\rho (q)a,qx)=\sum_{i=1,...,s}(\det
q)^{-w}\lambda ^{d_{i}}C_{i}(a,x).
\end{equation*}%
Hence, for \ $i=1,...,s,$, we get

\begin{equation*}
C_{i}(\rho (q)a,qx)=(\det q)^{-w}C_{i}(a,x).
\end{equation*}

The set of the homogeneous centro-affine covariants of degree $d$ denoted by 
$\Bbbk \lbrack \mathcal{C}(n,\Bbbk ,\Omega )\times \Bbbk
^{n}]_{d}^{GL(n,\Bbbk )}$ is a vectorial subspace of the $\Bbbk $-algebra $%
\Bbbk \lbrack \mathcal{C}(n,\Bbbk ,\Omega )\times \Bbbk ^{n}]^{GL(n,\Bbbk )}$
and for $d,d^{\prime }\in 
\mathbb{N}
$, we have

\begin{equation*}
\Bbbk \lbrack \mathcal{C}(n,\Bbbk ,\Omega )\times \Bbbk
^{n}]_{d}^{GL(n,\Bbbk )}\Bbbk \lbrack \mathcal{C}(n,\Bbbk ,\Omega )\times
\Bbbk ^{n}]_{d^{\prime }}^{GL(n,\Bbbk )}\subset \Bbbk \lbrack \mathcal{C}%
(n,\Bbbk ,\Omega )\times \Bbbk ^{n}]_{d+d^{\prime }}^{GL(n,\Bbbk )}.
\end{equation*}%
Hence, the $\Bbbk -$algebra $\Bbbk \lbrack \mathcal{C}(n,\Bbbk ,\Omega
)\times \Bbbk ^{n}]^{GL(n,\Bbbk )}$ can be written

\begin{equation*}
\Bbbk \lbrack \mathcal{C}(n,\Bbbk ,\Omega )\times \Bbbk ^{n}]^{GL(n,\Bbbk
	)}=\bigoplus {}_{d\in \mathbf{N}}\Bbbk \lbrack \mathcal{C}(n,\Bbbk ,\Omega
)\times \Bbbk ^{n}]_{d}^{GL(n,\Bbbk )}
\end{equation*}
and $\Bbbk \lbrack \mathcal{C}(n,\Bbbk ,\Omega )\times \Bbbk
^{n}]_{d}^{GL(n,\Bbbk )}$ is a direct sum of the vectorial subspaces $%
\mathcal{A}_{(d_{r_{1}},...,d_{r_{i}},\delta )}$ where $%
d_{r_{1}},...,d_{r_{i}},\delta \in 
\mathbb{N}
$

\begin{equation*}
\Bbbk \lbrack \mathcal{C}(n,\Bbbk ,\Omega )\times \Bbbk
^{n}]_{d}^{GL(n,\Bbbk )}=\bigoplus {}_{\substack{ d_{r_{1}},...,d_{r_{i}},%
		\delta \in 
		\mathbb{N}
		\\ \delta +\sum_{r_{i}\in \Omega }d_{r_{i}}=d}}\mathcal{A}%
_{(d_{r_{1}},...,d_{r_{i}},\delta )}.
\end{equation*}%
It's easy to check that

\begin{equation*}
\mathcal{A}_{(d_{r_{1}},...,d_{r_{i}},\delta )}\mathcal{A}%
_{(d_{r_{1}^{\prime }},...,d_{r_{i}^{\prime }},\delta ^{\prime })}\subseteq 
\mathcal{A}_{(d_{r_{1}},...,d_{r_{i}},\delta )+(d_{r_{1}^{\prime
	}},...,d_{r_{i}^{\prime }},\delta ^{\prime })}.
\end{equation*}%
Thus, the $\Bbbk -$algebra $\Bbbk \lbrack \mathcal{C}(n,\Bbbk ,\Omega
)\times \Bbbk ^{n}]^{GL(n,\Bbbk )}$ is multigraded and can be written

\begin{equation*}
\Bbbk \lbrack \mathcal{C}(n,\Bbbk ,\Omega )\times \Bbbk ^{n}]^{GL(n,\Bbbk
	)}=\bigoplus {}_{d_{r_{1}},...,d_{r_{i}},\delta \in 
	\mathbb{N}
}\mathcal{A}_{(d_{r_{1}},...,d_{r_{i}},\delta )}.
\end{equation*}%
Let $\mathcal{F}_{(d_{r_{1}},...,d_{r_{i}},\delta )}$ and $\mathcal{B}%
_{(d_{r_{1}},...,d_{r_{i}},\delta )}$ be respectively generating family and
a basis of the vectorial subspace $\mathcal{A}_{(d_{r_{1}},...,d_{r_{i}},%
	\delta )}$ of a given type $(d_{r_{1}},...,d_{r_{i}},\delta ).$

In view of Hilbert basis theorem, $\Bbbk -$algebra $\Bbbk \lbrack 
\mathcal{C}(n,\Bbbk ,\Omega )\times \Bbbk ^{n}]^{GL(n,\Bbbk )}$ is of finite
type and then the degrees of centro-affine covariants of a minimal system $%
\mathcal{B}(n,\Bbbk ,\Omega )$ of generators for $\mathcal{S}(n,\Bbbk
,\Omega )$ are bounded and let $D$ be the upper bound of degrees of these
generators.

The idea is to determine degree by degree a minimal system $\mathcal{B}%
(n,\Bbbk ,\Omega )$ for given differential systems $\mathcal{S}(n,\Bbbk
,\Omega )$ where constructing degree by degree the generating families $%
\mathcal{F}_{(d_{r_{1}},...,d_{r_{i}},\delta )}$, deduce bases $\mathcal{B}%
_{(d_{r_{1}},...,d_{r_{i}},\delta )}$ avoiding Aronhold's identities, then
determine a minimal system $\mathcal{B}(n,\Bbbk ,\Omega )$ with the help of
the test of membership in an ideal.

In view of the theorem of Gurevich, the generators of centro-affine
covariants for given differential systems $\mathcal{B}(n,\Bbbk ,\Omega )$ of
given type $(d_{r_{1}},...,d_{r_{i}},\delta ),r_{1},...,r_{i},\delta \in
\Omega $  obtained from the tensor \ref{tonsorfurmula} with the help of
successive alternations and complete contraction can be written as finite
products $f_{1}^{\alpha _{1}}...f_{l}^{\alpha _{l}},\alpha _{1},...,\alpha
_{l},l\in 
\mathbb{N}
$ of generators of lower degrees $f_{1},...,f_{l}$ where $%
(d_{r_{1}},...,d_{r_{l}},\delta )=\alpha
_{1}(d_{r_{1}}^{1},...,d_{r_{l}}^{1},\delta ^{1})+...+\alpha
_{l}(d_{r_{1}}^{l},...,d_{r_{l}}^{l},\delta ^{l})$ where $\alpha
_{1},...,\alpha _{1}\in 
\mathbb{N}
,$ $r_{1},...,r_{l}\in \Omega $ and for $i=1,...,l$, $f_{i}$ is of type $%
(d_{r_{i}}^{i},...,d_{r_{i}}^{i},\delta ^{i})$.

For instance, the generators of degree $d$ of centro-affine covariants for
systems $\mathcal{S}(2,\mathbb{R},\left \{ 1,1,2\right \} )$ are of type $%
(d_{0},d_{1},d_{2},\delta )$ where $d_{0}+d_{1}+d_{2}+\delta =d$ and $%
d_{2}-d_{0}-\delta \equiv 0[2]$ . Since every generating family contain
generators of lower degrees then start from the generating families $%
\mathcal{F}_{(d_{0},d_{1},d_{2},\delta )}$ of degree $1.$ The only
generators of degree $1$ are of type $(0,1,0,0)$ and can be obtained from
the tensor once contravariant and once covariant $a_{-}^{-}$ with the help
of successive tensorial operations of contraction on the
tensor from the tensor. Then obtain $ \mathcal{F}_{(0,1,0,0)}=\{ I_{1} \}$.

The only generators of degree $2$ are of type $(1,0,1,0)$, $(0,2,0,0)$, $(0,0,1,1)$
and can be obtained respectively from the tensors $%
a^{-}a_{--}^{-}$, $a_{-}^{-}a_{-}^{-}$ and $a_{--}^{-}x^{-} :$ 

\begin{equation*}
\begin{array}{lcll}
a^{-}a_{--}^{-} & \mapsto  & a^{\alpha }a_{\alpha \beta }^{\beta } & = I_{17}\\ 
a_{-}^{-}a_{-}^{-} & \mapsto  & a_{\alpha }^{\alpha }a_{\beta }^{\beta } & = I_{1}^{2} \\ 
&  & a_{\beta }^{\alpha }a_{\alpha }^{\beta } & = I_{2} \\ 
a^{-}a_{--}^{-} & \mapsto  & a_{\alpha \beta }^{\alpha }x^{\beta } & = K_{1}%
\end{array}%
\end{equation*}%
then obtain $\mathcal{F}_{(1,0,1,0)},\;\mathcal{F}_{(0,2,0,0)},\;\mathcal{F}%
_{(0,0,1,1)}$ where

\begin{itemize}[leftmargin=2cm] 
	\item[$\mathcal{F}_{(1,0,1,0)}=$] $\{ I_{17} \} $. 
	
	\item[$\mathcal{F}_{(0,2,0,0)}=$] $\{ I_{1}^{2} $, $I_{2} \} $.
	
	\item[$\mathcal{F}_{(0,0,1,1)}=$] $\{ K_{1} \}$.
\end{itemize}

In the same manner we found,

\begin{itemize}[leftmargin=2cm] 

\item[ $\mathcal{F}_{(0,1,1,1)}$ = ]$\{$$I_{1}K_{1}$, $K_{3}$, $K_{4}$$ \} $.
	
\item[ $\mathcal{F}_{(1, 2, 1, 0)}$ = ]$ \{$$I_{1}^{2}I_{17}$, $I_{1}I_{20}$, $I_{1}I_{19}$, $I_{2}I_{17}$, $I_{24}$ 	
	
\item[ $\mathcal{F}_{(1,1,3,0)}$ = ]$ \{$$I_{1}I_{26}$, $I_{1}I_{25}$, $I_{5}I_{17}$, $I_{4}I_{17}$, $I_{3}I_{17}$, $I_{32}$, $I_{31}$, $I_{30}$$ \} $.

\item[ $\mathcal{F}_{(2,3,2,0)}$ = ]$ \{$$I_{1}^{3}I_{17}^{2}$, $I_{1}^{3}I_{23}$, $I_{1}^{3}I_{22}$, $I_{1}^{2}I_{17}I_{20}$, $I_{1}^{2}I_{17}I_{19}$, $I_{1}^{2}I_{29}$, $I_{1}^{2}I_{28}$, $I_{1}I_{2}I_{17}^{2}$, $I_{1}I_{2}I_{23}$, $I_{1}I_{2}I_{22}$,  
$I_{1}I_{5}I_{18}$, $I_{1}I_{4}I_{18}$, $I_{1}I_{3}I_{18}$, $I_{1}I_{17}I_{24}$, $I_{1}I_{20}^{2}$, $I_{1}I_{19}I_{20}$, $I_{1}I_{19}^{2}$, $I_{2}I_{17}I_{20}$, $I_{2}I_{17}I_{19}$, $I_{2}I_{29}$,  
$I_{2}I_{28}$, $I_{6}I_{18}$, $I_{20}I_{24}$, $I_{19}I_{24}$$ \} $ .
	
\item[ $\mathcal{F}_{(2,2,2,0)}$ = ]$ \{$$I_{1}^{2}I_{17}^{2}$, $I_{1}^{2}I_{23}$, $I_{1}^{2}I_{22}$, $I_{1}I_{17}I_{20}$, $I_{1}I_{17}I_{19}$, $I_{1}I_{29}$, $I_{1}I_{28}$, $I_{2}I_{17}^{2}$, $I_{2}I_{23}$, $I_{2}I_{22}$,  
$I_{5}I_{18}$, $I_{4}I_{18}$, $I_{3}I_{18}$, $I_{17}I_{24}$, $I_{20}^{2}$, $I_{19}I_{20}$, $I_{19}^{2}$ $ \} $.

\end{itemize}

where the family $\left\{I_{1},\ldots ,I_{36},K_{1},\ldots ,K_{33}\right\} $ is a minimal system of generators of centro-affine invariants and covariants of $ \mathcal{S}(2,\mathbb{R},\left \{ 0,1,2\right \} ) $\cite{Sibirskii,BoularasandDali,Vulpe}.

For a given type $(d_{r_{1}},...,d_{r_{i}},\delta ),r_{1},...,r_{i}\in \Omega ,$ $%
d_{r_{1}},...,d_{r_{i}},\delta \in 
\mathbb{N}
,$ a basis $\mathcal{B}_{(d_{r_{1}},...,d_{r_{i}},\delta )}$ is a subfamily $%
\left \{ f_{s_{1}},\cdot \cdot \cdot ,f_{s_{t}}\right \} ,1\leq t\leq s$ of $%
\mathcal{F}_{(d_{r_{1}},...,d_{r_{i}},\delta )}$ where $f_{s_{1}},\cdot
\cdot \cdot ,f_{s_{t}}$ are linearly independent and $\left \langle \left \{
f_{l_{1}}\cdot \cdot \cdot ,f_{l_{t}}\right \} \right \rangle =\left \langle
\left \{ f_{1}\cdot \cdot \cdot ,f_{s}\right \} \right \rangle $. We shall show
how the elements $f_{s_{1}},\cdot \cdot \cdot ,f_{s_{t}}$ of $\mathcal{B}%
_{(d_{r_{1}},...,d_{r_{i}},\delta )}$ can be determined. Denoting by $%
(t_{r}^{1})^{p_{r}}$ the product $\prod \limits_{i=1,...,\nu
}((t_{r}^{1})^{i})^{p_{r}^{i}}$ and by $x^{\alpha }$ the product $%
(x^{1})^{\alpha _{1}}...(x^{n})^{\alpha _{n}}$ where $\nu =\dim \mathcal{T}%
_{r}^{1},$ $r\in \Omega ,$ $p_{r},\alpha \in 
\mathbb{N}
,$ $(p_{r}^{i},...,p_{r}^{i})$ and $(\alpha _{1},...,\alpha _{n})$ are the
partitions respectively of $p_{r}$ and $\alpha .$ 

A monomial associated with 
$\mathcal{S}(n,\Bbbk ,\Omega )$ is a finite product of the form

\begin{equation*}
(t_{r_{1}}^{1})^{p_{r_{1}}}...(t_{r_{i}}^{1})^{p_{r_{i}}}x^{\alpha
},r_{1},...,r_{i}\in \Omega \text{ and }p_{r_{1}},...,p_{r_{i}}\in 
\mathbb{N}%
\end{equation*}%
and can be written simply

\begin{equation}
 (a_{\alpha _{1}...\alpha _{r_{1}}}^{i})^{p_{r_{1}}}...(a_{\alpha
	_{1}...\alpha _{r_{i}}}^{j})^{p_{r_{i}}}x^{\alpha }   \label{monom}
\end{equation}%
where for $i,$ $j,\alpha _{1}\cdots \alpha _{r_{i}}=1,...,n,a^{j},a_{\alpha
	_{1}...\alpha _{r_{i}}}^{j}\in \mathcal{C}(n,\Bbbk ,\Omega )$. If we define

\begin{equation*}
(a_{\alpha _{1}...\alpha _{r_{1}}}^{i})^{p_{r_{1}}}...(a_{\alpha
	_{1}...\alpha _{r_{i}}}^{j})^{p_{r_{i}}}x^{\alpha }\times (a_{\alpha
	_{1}...\alpha _{r_{1}}}^{i})^{q_{r_{1}}}... (a_{\alpha _{1}...\alpha
	_{r_{i}}}^{j}) ^{q_{r_{i}}}x^{\mu }=(a_{\alpha _{1}...\alpha
	_{r_{1}}}^{i})^{p_{r_{1}}+q_{r_{1}}}... (a_{\alpha _{1}...\alpha
	_{r_{i}}}^{j}) ^{p_{r_{i}}+q_{r_{i}}}x^{\alpha +\mu }
\end{equation*}%

where $p_{r_{1}},\ldots ,p_{r_{i}},q_{r_{1}},\ldots ,q_{r_{i}}\in 
\mathbb{N}
^{m}$ with $m= 2(r_{i{\tiny {\tiny }}}+1)$, then the set of all monomials (\ref{monom}) denoted by $%
\mathcal{M}$ is a monoid with the identity $1$. A monomial (\ref{monom}) of
a centro-affine covariant of given type $(d_{r_{1}},...,d_{r_{i}},\delta )$
can be written

\begin{equation}
(a_{\alpha _{1}...\alpha _{r_{1}}}^{i})^{d_{r_{1}}}...(a_{\alpha
	_{1}...\alpha _{r_{i}}}^{j})^{d_{r_{i}}}x^{\delta }  \label{monomtyp}
\end{equation}%
and will be called a monomial of type $(d_{r_{1}},...,d_{r_{i}},\delta ). $%
\ The set of  the monomials (\ref{monomtyp}) of type $%
(d_{r_{1}},...,d_{r_{i}},\delta )$ denoted by $\mathcal{M}%
_{(d_{r_{1}},...,d_{r_{i}},\delta )}$ is finite and can be  written

\begin{equation*}
\mathcal{M}_{(d_{r_{1}},...,d_{r_{i}},\delta )}=\left \{ m_{1},\cdot \cdot
\cdot ,m_{l},m_{1}\prec m_{2}...\prec m_{l}\right \} 
\end{equation*}%
where $l$ is the cardinality of $\mathcal{M}_{(d_{r_{1}},...,d_{r_{i}},%
	\delta )}$ and $\prec $ is some monomial order on $\mathcal{M}$. Hence an
element $f$ of generating the family $\mathcal{F}_{(d_{r_{1}},...,d_{r_{i}},%
	\delta )}$ of the given type $(d_{r_{1}},...,d_{r_{i}},\delta )$ can be
decomposed in $\mathcal{M}_{(d_{r_{1}},...,d_{r_{i}},\delta )}$ as follows%
\begin{equation*}
f=\alpha _{1}m_{1}+\alpha _{2}m_{2}+\cdot \cdot \cdot +\alpha
_{l}m_{l},\alpha _{1},\alpha _{2},\cdot \cdot \cdot ,\alpha _{l}\in \Bbbk 
\end{equation*}%
The vector $v=\left( \alpha _{1},\alpha _{2},\cdot \cdot \cdot ,\alpha
_{l}\right) \in \Bbbk ^{l}$ is called the vector associated with $f$. Now
let $v_{1}\cdot \cdot \cdot ,v_{s}$ be the vectors associated respectively
with $f_{1},\cdot \cdot \cdot ,f_{s}$ then $rank\left \{ f_{1},\cdot \cdot
\cdot ,f_{s}\right \} =$ $rank\left \{ v_{1}\cdot \cdot \cdot ,v_{s}\right \} $.
Hence, the family $\left \{ v_{l_{1}}\cdot \cdot \cdot ,v_{l_{t}}\right \} $
is a basis of the vectorial subspace generated by the vectors $v_{1}\cdot
\cdot \cdot ,v_{s}$ if and only if $\mathcal{B}_{(d_{r_{1}},...,d_{r_{i}},%
	\delta )}=\left \{ f_{l_{1}}\cdot \cdot \cdot ,f_{l_{t}}\right \} $. The
vectors $v_{1}\cdot \cdot \cdot ,v_{s}$ can be determined using the
algorithm 2 \cite{DahiraDaliandSuiSunCheng}. We illustrate our idea by means of examples. To
determine $\mathcal{B}_{(0,2,0,0)}$ for the planar cubic differential
systems $\mathcal{S}(2,\Bbbk ,\left \{ 0,1,2,3\right \} )$ one construct $%
\mathcal{F}_{(0,0,2,0)}$ where constructing the centro-affines covariants of
type $(0,0,2,0)$. That is, $\mathcal{F}%
_{(0,0,2,0)}=\left \{ J_{4},J_{5}\right \} $ where 
\begin{equation*}
\begin{array}{cc}
J_{4} & =a_{\alpha pr}^{\alpha }a_{\beta qs}^{\beta }\varepsilon
^{pq}\varepsilon ^{rs} \\ 
J_{5} & =a_{\beta pr}^{\alpha }a_{\alpha qs}^{\beta }\varepsilon
^{pq}\varepsilon ^{rs}%
\end{array}%
,\alpha ,\beta ,p,q=1,2
\end{equation*}%
since $J_{4},J_{5}$ are the lonely centro-affine covariants obtained from
the tensor $a_{---}^{-}a_{---}^{-}$
 with the help of successive alternations and complete contraction.
The set $\mathcal{M}_{(0,0,2,0)}$ of all monomials of type $(0,0,2,0)$ can
be determined where expanding the elements of $\mathcal{F}_{(0,2,0,0)}$:

\[
\begin{array}{cc}
J_{4}= & 2a_{111}^{1}a_{122}^{1}+2\ensuremath{a_{111}^{1}a_{222}^{2}}-2\ensuremath{(a_{112}^{1})^{2}-4\ensuremath{a_{112}^{1}a_{122}^{2}+2a_{122}^{1}a_{112}^{2}}}+2\ensuremath{a_{112}^{2}a_{222}^{2}}-2\ensuremath{(a_{122}^{2})^{2}}\\
J_{5}= & 2a_{111}^{1}a_{122}^{1}-2\ensuremath{(a_{112}^{1})^{2}}+2\ensuremath{a_{112}^{1}a_{122}^{2}}-4\ensuremath{a_{122}^{1}a_{112}^{2}+2a_{222}^{1}a_{111}^{2}+2\ensuremath{a_{112}^{2}a_{222}^{2}}-2\ensuremath{(a_{122}^{2})^{2}}}
\end{array}\text{}
\]
then 
\[
\mathcal{M}_{(0,0,2,0)}=\left\{ m_{1},m_{2},m_{3},m_{4},m,m_{6},m_{7},m_{8},\;m_{1}\prec m_{2}\prec m_{3}\prec m_{4}\prec m\prec m_{6}\prec m_{7}\prec m_{8}\right\} 
\]
where the monomials 
$m_{1}=a_{111}^{1}a_{122}^{1}$, $m_{2}=a_{111}^{1}a_{222}^{2}$, $ m_{3}=\ensuremath{(a_{112}^{1})^{2}}$, $m_{4}=\ensuremath{a_{112}^{1}a_{122}^{2}}$, $m_{5}=\ensuremath{a_{122}^{1}a_{112}^{2}}$, $m_{6}=\ensuremath{a_{222}^{1}a_{111}^{2}}$, $m_{7}=\ensuremath{a_{112}^{2}a_{222}^{2}}$, $m_{8}=\ensuremath{(a_{122}^{2})^{2}}$

and

\[
\begin{array}{cc}
J_{4}= & 2m_{1}+2m_{2}-2m_{3}-4m_{4}+2m_{5}+0m_{6}+2m_{7}-2m_{8}\\
J_{5}= & 2m_{1}+0m_{2}-2m_{2}+2m_{4}-4m_{5}+2m_{6}+2m_{7}-2m_{8}
\end{array}
\]
Hence,

\[
\begin{array}{cc}
v_{1}= & (2,2,-2,-4,2,0,2,-2)\\
v_{2}= & (2,0,-2,2,-4,2,2,-2)
\end{array}
\]
Then obtain $\mathcal{B}_{(0,0,2,0)}=\left \{ J_{4},J_{5}\right \} $ since
the vectors $v_{1}$ and $v_{2}$ are linearly independent.

We can determine in the same manner for $\mathcal{B}_{(0,0,4,0)}$ for $%
\mathcal{S}(2,%
\mathbb{R}
,\left \{ 0,1,2,3\right \} ).$ We find the generators of centro-affine
covariants $\mathcal{F}_{(0,0,4,0)}=\left \{
(J_{4})^{2},J_{4}J_{5},(J_{5})^{2},J_{19}\right \} $ where 
\begin{equation*}
\begin{array}{cc}
(J_{4})^{2} & =a_{\alpha pr}^{\alpha }a_{\beta qs}^{\beta }a_{\gamma
	km}^{\gamma }a_{\delta ln}^{\delta }\varepsilon ^{pq}\varepsilon
^{rs}\varepsilon ^{kl}\varepsilon ^{mn} \\ 
J_{4}J_{5} & =a_{\alpha pr}^{\alpha }a_{\beta qs}^{\beta }a_{\delta
	km}^{\gamma }a_{\gamma ln}^{\delta }\varepsilon ^{pq}\varepsilon
^{rs}\varepsilon ^{kl}\varepsilon ^{mn} \\ 
(J_{5})^{2} & =a_{\beta pr}^{\alpha }a_{\alpha qs}^{\beta }a_{\delta
	km}^{\gamma }a_{\gamma ln}^{\delta }\varepsilon ^{pq}\varepsilon
^{rs}\varepsilon ^{kl}\varepsilon ^{mn} \\ 
J_{19} & =a_{\beta pr}^{\alpha }a_{\delta qs}^{\beta }a_{\alpha km}^{\gamma
}a_{\gamma ln}^{\delta }\varepsilon ^{pq}\varepsilon ^{rs}\varepsilon
^{kl}\varepsilon ^{mn}%
\end{array}%
,
\end{equation*}%
expand $(J_{4})^{2},J_{4}J_{5},(J_{5})^{2}$ and $J_{19}$ then obtain all the
monomials of type $(0,0,4,0)$: $(a^{1}_{111})^{2}(a^{1}_{122})^{2}$, 
$(a^{1}_{111})^{2}a^{1}_{122}a^{2}_{222}$, 
$(a^{1}_{111})^{2}a^{1}_{222}a^{2}_{122}$, 
$(a^{1}_{111})^{2}(a^{2}_{222})^{2}$, 
$a^{1}_{111}(a^{1}_{112})^{2}a^{1}_{122}$, 
$a^{1}_{111}(a^{1}_{112})^{2}a^{2}_{222}$, 
$a^{1}_{111}a^{1}_{112}a^{1}_{122}a^{2}_{122}$,  
$a^{1}_{111}a^{1}_{112}a^{1}_{222}a^{2}_{112}$,  
$a^{1}_{111}a^{1}_{112}a^{2}_{122}a^{2}_{222}$, 
$a^{1}_{111}(a^{1}_{122})^{2}a^{2}_{112}$,  
$a^{1}_{111}a^{1}_{122}a^{1}_{222}a^{2}_{111}$,  
$a^{1}_{111}a^{1}_{122}a^{2}_{112}a^{2}_{222}$,
 $a^{1}_{111}a^{1}_{122}(a^{2}_{122})^{2}$,
$a^{1}_{111}a^{1}_{222}a^{2}_{111}a^{2}_{222}$,  
$a^{1}_{111}a^{1}_{222}a^{2}_{112}a^{2}_{122}$,  
$a^{1}_{111}a^{2}_{112}(a^{2}_{222})^{2}$,   
$a^{1}_{111}(a^{2}_{122})^{2}a^{2}_{222}$,  
$(a^{1}_{112})^{4}$, $(a^{1}_{112})^{3}a^{2}_{122}$,   
$(a^{1}_{112})^{2}a^{1}_{122}a^{2}_{112}$,  
$(a^{1}_{112})^{2}a^{1}_{222}a^{2}_{111}$,  
$(a^{1}_{112})^{2}a^{2}_{112}a^{2}_{222}$,  
$(a^{1}_{112})^{2}(a^{2}_{122})^{2}$,  
$a^{1}_{112}(a^{1}_{122})^{2}a^{2}_{111}$, 
$a^{1}_{112}a^{1}_{122}a^{2}_{111}a^{2}_{222}$, 
$a^{1}_{112}a^{1}_{122}a^{2}_{112}a^{2}_{122}$,  
$a^{1}_{112}a^{1}_{222}a^{2}_{111}a^{2}_{122}$,  
$a^{1}_{112}a^{1}_{222}(a^{2}_{112})^{2}$, 
$a^{1}_{112}a^{2}_{111}(a^{2}_{222})^{2}$, 
$a^{1}_{112}a^{2}_{112}a^{2}_{122}a^{2}_{222}$,  
$a^{1}_{112}(a^{2}_{122})^{3}$,  
$(a^{1}_{122})^{2}a^{2}_{111}a^{2}_{122}$, 
$(a^{1}_{122})^{2}(a^{2}_{112})^{2}$, 
$a^{1}_{122}a^{1}_{222}a^{2}_{111}a^{2}_{112}$, 
$a^{1}_{122}a^{2}_{111}a^{2}_{122}a^{2}_{222}$, 
$a^{1}_{122}(a^{2}_{112})^{2}a^{2}_{222}$,  
$a^{1}_{122}a^{2}_{112}(a^{2}_{122})^{2}$, 
$(a^{1}_{222})^{2}(a^{2}_{111})^{2}$, 
$a^{1}_{222}a^{2}_{111}a^{2}_{112}a^{2}_{222}$, 
$a^{1}_{222}a^{2}_{111}(a^{2}_{122})^{2}$, 
$a^{1}_{222}(a^{2}_{112})^{2}a^{2}_{122}$, 
$(a^{2}_{112})^{2}(a^{2}_{222})^{2}$, 
$a^{2}_{112}(a^{2}_{122})^{2}a^{2}_{222}$, 
$(a^{2}_{122})^{4}$, then obtain $\mathcal{M}%
_{(0,0,4,0)}$ the set of all monomials of the type $(0,0,4,0)$ given in some
total ordering $\prec .$

\begin{equation*}
\mathcal{M}_{(0,4,0)}=\left \{ m_{1},\ldots ,m_{44},\; m_{1}\prec ...\prec m_{44} \right \}
\end{equation*}%
calculate the vectors associated with $(J_{4})^{2},J_{4}J_{5},(J_{5})^{2}$
and $J_{19}:$

\[
\begin{array}{ll}
w_{1}=& (4,8,0,4,-8,-8,-16,0,-16,8,0,16,-8,0,0,8,-8,4,16,-8,0,-8,24,0,0,-16,0,0,0,-16,16,...,4,-8,4)\\
w_{2}=& (4,4,0,0,-8,-4,-4,0,4,-4,4,0,-8,4,0,4,-4,4,4,4,-4,-8,0,0,0,20,-8,0,0,-4,4,...,4,-8,4)\\
w_{3}=& (4,0,0,0,-8,0,8,0,0,-16,8,8,-8,0,0,0,0,4,-8,16,-8,-8,12,0,0,-16,8,0,0,8,-8,...,4,-8,4)\\
w_{4}=& (4,0,2,0,-8,0,2,-4,2,-8,6,0,-4,2,0,0,0,4,-4,12,-4,-4,2,-2,0,2,4,-4,2,2,-4,...,4,-8,4)
\end{array}
\]

Then $\mathcal{B}_{(0,0,4,0)}=\left \{
(J_{4})^{2},J_{4}J_{5},(J_{5})^{2},J_{19}\right \} $ since $w_{1},w_{2},w_{3}$
and $w_{4}$ are linearly independent. For $d\geq 1$, let $\mathcal{B}_{d}$
be the set of the centro-affine covariants of $\mathcal{B}(n,\Bbbk ,\Omega )$
of degree less than or equal to $d$ and $\mathcal{I}_{d}=\left \langle 
\mathcal{B}_{d}\right \rangle $ the ideal generated by $\mathcal{B}_{d}$. In
view of the Hilbert basis theorem the sequence of the ideals $\left( 
\mathcal{I}_{d}\right) _{d\geq 1}$ is increasing and stationary. Hence, $%
\mathcal{B}(n,\Bbbk ,\Omega )$ is finite and therefore $\bigoplus%
\limits_{d_{0}+d_{1}+d_{3}+\delta \geq 1}$\ $\mathcal{B}%
_{(d_{0},d_{1},d_{3},\delta )}$ is finite. 

There exist an upper bound of
degrees of these generators which we denote by $D$. This bound has been
calculated by V. Popov \cite{Popov}  and recently improved by H. Derksen \cite{HDerksenHKraft} but still too large. We find $D$ when all the centro-affine covariants of degree $d\geq D$ are reducible. Now we are able to determine a minimal
system $\mathcal{B}(n,\Bbbk ,\Omega ).$ 

\begin{theorem}
Let $D$ be the upper bound of degrees of the generators of the centro-affine covariants differential systems $\mathcal{S}(n,\Bbbk ,\Omega ).$
	
	if $\bigoplus \limits_{1\leq d_{0}+d_{1}+d_{3}+\delta \leq D}$\ $\mathcal{B}%
	_{(d_{0},d_{1},d_{3},\delta )}=\{J_{1}^{\alpha _{1}^{1}}...J_{\tau }^{\alpha
		_{\tau }^{1}},...,J_{1}^{\alpha _{1}^{s}}...J_{\tau }^{\alpha _{\tau }^{s}}\}
	$, $i=1,...,s$ and $j=1,...,\tau $ $\alpha _{j}^{i}\in 
	\mathbb{N}
	$, then 
	$$  \mathcal{B}(n,\Bbbk ,\Omega ) 
	= \{J_{1},...,J_{\tau }\}  .$$
\end{theorem}

\begin{proof}
	We have,
	\begin{equation*}
	\Bbbk \lbrack \mathcal{C}(n,\Bbbk ,\Omega )\times \Bbbk ^{n}]^{GL(n,\Bbbk
		)}=\bigoplus {}_{d\in \mathbf{N}}\Bbbk \lbrack \mathcal{C}(n,\Bbbk ,\Omega
	)\times \Bbbk ^{n}]_{d}^{GL(n,\Bbbk )}
	\end{equation*}%
	and 
	\begin{equation*}
	\Bbbk \lbrack \mathcal{C}(n,\Bbbk ,\Omega )\times \Bbbk
	^{n}]_{d}^{GL(n,\Bbbk )}=\bigoplus {}_{\substack{ d_{r_{1}},...,d_{r_{i}},%
			\delta \in 
			\mathbb{N}
			\\ \delta +\sum_{r_{i}\in \Omega }d_{r_{i}}=d}}\mathcal{A}%
	_{(d_{r_{1}},...,d_{r_{i}},\delta )},
	\end{equation*}%
	then%
	\begin{equation*}
	\Bbbk \lbrack \mathcal{C}(n,\Bbbk ,\Omega )\times \Bbbk ^{n}]^{GL(n,\Bbbk
		)}=\left \langle \{J_{1}^{\alpha _{1}^{1}}...J_{\tau }^{\alpha _{\tau
		}^{1}},...,J_{1}^{\alpha _{1}^{s}}...J_{\tau }^{\alpha _{\tau
		}^{s}}\} \right \rangle 
	\end{equation*}
where for $i=1,...,\tau ,d%
{{}^\circ}%
J_{\tau }^{\alpha _{\tau }^{1}}\leq D.$

 Since $J_{1}^{\alpha _{1}^{1}}...J_{\tau }^{\alpha _{\tau
	}^{1}},...,J_{1}^{\alpha _{1}^{s}}...J_{\tau }^{\alpha _{\tau }^{s}}$ are
linearly independents and $J_{1}^{\alpha _{1}^{1}}...J_{\tau }^{\alpha
	_{\tau }^{1}},...,J_{1}^{\alpha _{1}^{s}}...J_{\tau }^{\alpha _{\tau
	}^{s}}\in $ $\left \langle \{J_{1},...,J_{\tau }\} \right \rangle $, thus%
 
\begin{equation*}
\mathcal{B}(n,\Bbbk ,\Omega )=\{J_{1},...,J_{\tau }\}
\end{equation*}

The theorem is proved.
\end{proof}

For instance, consider the planar differential systems $\mathcal{B}(2,%
\mathbb{R}
,\left \{ 0,1\right \} )$, that is,

\begin{equation*}
\frac{dx^{j}}{dt}=a^{j}+a_{\alpha }^{j}x^{\alpha }\ j,\alpha =1,2
\end{equation*}%
Degree by degree we find the generating families $\mathcal{F}%
_{(d_{0},d_{1},d_{3},\delta )}$  of types $(d_{0},d_{1},d_{3},\delta )$ 
where $d_{0},d_{1},d_{3},\delta \in 
\mathbb{N}$
and $d_{0}+d_{1}+d_{3}+\delta \leq d$ of given degree $d$ and with the
help of the algorithm 2 deduce  corresponding bases $\mathcal{B}%
_{(d_{0},d_{1},d_{3},\delta )}$. Therefore

\begin{equation*}
\bigoplus \limits_{d_{0}+d_{1}+\delta \geq 1}\  \mathcal{B}%
_{(d_{0},d_{1},d_{3},\delta
	)}=\{I_{1},I_{1}^{2},I_{2},K_{21},I_{1}^{3},I_{1}I_{2},I_{18},K_{21}I_{1},K_{22},K_{2}\}
\end{equation*}%
since all the centro-affine covariants of greater degree are
reducible, where 
\begin{equation*}
\begin{array}{ll}
I_{1}=a_{\alpha }^{\alpha }; & K_{21}=a^{p }x^{q };\\
I_{2}=a_{\alpha }^{\beta }a_{\beta }^{\alpha };  & K_{2}=a_{\alpha }^{p}x^{\alpha }x^{q}\varepsilon _{pq};\\
I_{18}=a_{\alpha}^{p}a^{\alpha }a^{q}\varepsilon _{pq}&K_{22}=a_{\alpha }^{p}a^{\alpha }x^{q}\varepsilon

\end{array}%
\end{equation*}%
then one get
$$\mathcal{B}(2,%
\mathbb{R}
,\left \{ 0,1\right \} )=\{I_{1},I_{2},I_{18},K_{2},K_{21},K_{22}\}.$$

\section{Minimal system of generators of centro-affine invariants and
	covariants for $\mathcal{S}(2,\mathbb{R},\left \{ 0,1,3\right \} )$}

In this section we describe the algebra of centro-affine covariants for
differential systems $\mathcal{S}(2,%
\mathbb{R}
,\Omega _{0})$ where $\Omega _{0}=\left \{ 0,1,3\right \} $

\begin{equation}
\frac{dx^{j}}{dt}=a^{j}+a_{\alpha }^{i}x^{\alpha }+a_{\alpha \beta \gamma
}^{j}x^{\alpha }x^{\beta }x^{\gamma },\ j,\alpha ,\beta ,\gamma =1,2\label{syscubic}
\end{equation}%
A monomial (\ref{monom}) associated with $\mathcal{S}(2,%
\mathbb{R}
,\left \{ 0,1,3\right \} )$ is a finite product of the form

\begin{equation}
(a^{j})^{p_{0}}(a_{\alpha }^{j})^{p_{1}}(a_{\alpha \beta \gamma
}^{j})^{p_{3}}(x)^{\alpha }.  \label{cubicmonom}
\end{equation}%
where $p_{0}\in 
\mathbb{N}
^{2},p_{1}\in 
\mathbb{N}
^{4},p_{3}\in 
\mathbb{N}
^{8},\alpha \in 
\mathbb{N}
^{2}$. Let us\textit{\ }order the coefficients of $\mathcal{C}(2,%
\mathbb{R}
,\left \{ 0,1,3\right \} )$ and the components $x^{1},x^{2}$ of the
contravariant vector\ $x$ in the following manner: 
\begin{equation}
a^{1}\prec a^{2}\prec a_{1}^{1}\prec a_{2}^{1}\prec a_{1}^{2}\prec
a_{2}^{2}\prec a_{111}^{1}\prec a_{112}^{1}\prec a_{122}^{1}\prec
a_{222}^{1}\prec a_{111}^{2}\prec a_{112}^{2}\prec a_{122}^{2}\prec
a_{222}^{2}\prec x^{1}\prec x^{2}.  \label{cubicorder}
\end{equation}
The total ordering defined by (\ref{cubicmonom}) is a total ordering and can
be extended to a total lexicographic ordering for the set $\mathcal{M}$ of
all the monoials (\ref{cubicmonom}) and in the usual manner (see e.g. [12,
pp. 373-375]) 

\begin{equation*}
\begin{array}{c}
(a^{j})^{p}(a_{\alpha }^{j})^{r}(a_{\alpha \beta \gamma }^{j})^{q}x^{r}\prec
(a^{j})^{p^{\prime }}(a_{\alpha }^{j})^{r^{\prime }}(a_{\alpha \beta \gamma
}^{j})^{q^{\prime }}x^{r^{\prime }} \\ 
\Leftrightarrow  \\ 
\text{the first nonzero component of the vector }(p-p^{\prime },r-r^{\prime
},q-q^{\prime },r-r^{\prime })\text{ is positive}%
\end{array}%
\end{equation*}%
$\mathcal{\ }$ Given type $(d_{0},d_{1},d_{3},\delta )$ the set of all the
monomials of type $(d_{0},d_{1},d_{3},\delta )$ can be written as $\mathcal{M%
}_{(d_{0},d_{1},d_{3},\delta )}=\left \{ m_{1},...,m_{l},m_{1}\prec
m_{2}...\prec m_{l}\right \} $. Hence, given type $(d_{0},d_{1},d_{3},\delta )
$ one use the algorithm 2 \cite{DahiraDaliandSuiSunCheng} to decompose the elements  $%
f_{1},...,f_{\tau }$ of $\mathcal{F}_{(d_{0},d_{1},d_{3},\delta )}$ in $%
\mathcal{M}_{(d_{0},d_{1},d_{3},\delta )}$ then obtain $v_{1},...,v_{\tau }$
the vectors associated with $f_{1},...,f_{\tau }$ respectively then deduce $%
\mathcal{B}_{(d_{0},d_{1},d_{3},\delta )},d_{0}+d_{1}+d_{3}+\delta \geq 1$
then get  a minimal system $\mathcal{B}(2,%
\mathbb{R}
,\left \{ 0,1,3\right \} )$ of generators of centro-affine invariants of
degree up to $9$  for $\mathcal{S}(2,%
\mathbb{R}
,\left \{ 0,1,3\right \} ).$

\begin{theorem}
	The family $\mathcal{S}(\Omega )=\{J_{1},\ldots ,J_{47},K_{1},\ldots
	,K_{75}\}$ form a minimal system of generators of centro-affine invariants and covariants of $\mathcal{S}(2,\mathbb{R},\left \{ 0,1,3\right \} )$, where
	
	\begin{center}
		\begin{tabular}{lll}
			$J_{1}=a_{\alpha}^{\alpha}$ & $J_{4}=a_{\alpha pr}^{\alpha}a_{\beta qs}^{\beta}\varepsilon^{pq}\varepsilon^{rs}$ & $J_{7}=a^{\alpha}a^{\beta}a_{\gamma\alpha\beta}^{\gamma}$\tabularnewline
			$J_{2}=a_{\beta}^{\alpha}a_{\alpha}^{\beta}$ & $J_{5}=a_{\beta pr}^{\alpha}a_{\alpha qs}^{\beta}\varepsilon^{pq}\varepsilon^{rs}$ & $J_{8}=a_{p}^{\alpha}a_{\gamma}^{\beta}a_{\beta\alpha q}^{\gamma}\varepsilon^{pq}$\tabularnewline
			$J_{3}=a_{p}^{\alpha}a_{\beta\alpha q}^{\beta}\varepsilon^{pq}$ & $J_{6}=a^{\alpha}a^{q}a_{\alpha}^{p}\varepsilon_{pq}$ & $J_{9}=a_{\gamma}^{\alpha}a_{\alpha pr}^{\beta}a_{\beta qs}^{\gamma}\varepsilon^{pq}\varepsilon^{rs}$\tabularnewline
		\end{tabular}
		\par\end{center}
	
	\begin{center}
		\begin{tabular}{cc}
			$J_{10}=a_{\gamma pr}^{\alpha}a_{\alpha qk}^{\beta}a_{\beta sl}^{\gamma}\varepsilon^{pq}\varepsilon^{rs}\varepsilon^{kl}$ & $J_{11}=a_{\gamma\beta p}^{\alpha}a_{\alpha rk}^{\beta}a_{qsl}^{\gamma}\varepsilon^{pq}\varepsilon^{rs}\varepsilon^{kl}$\tabularnewline
		\end{tabular}
		\par\end{center}
	
	\begin{center}
		\begin{tabular}{ll}
			$J_{12}=a^{\alpha}a^{\beta}a_{\delta}^{\gamma}a_{\beta\gamma\alpha}^{\delta}$ & $J_{20}=a^{\alpha}a^{\beta}a^{\gamma}a^{p}a_{\alpha\beta\gamma}^{q}\varepsilon_{pq}$\tabularnewline
			$J_{13}=a^{\alpha}a^{\beta}a_{\beta}^{\gamma}a_{\delta\gamma\alpha}^{\delta}$ & $J_{21}=a^{\alpha}a^{\beta}a_{\mu}^{\gamma}a_{\beta}^{\delta}a_{\gamma\delta\alpha}^{\mu}$\tabularnewline
			$J_{14}=a^{\alpha}a^{\beta}a_{\alpha\beta p}^{\gamma}a_{\delta\gamma q}^{\delta}\varepsilon^{pq}$ & $J_{22}=a^{\alpha}a^{\beta}a_{\mu}^{\gamma}a_{\delta\gamma p}^{\delta}a_{\beta\alpha q}^{\mu}\varepsilon^{pq}$\tabularnewline
			$J_{15}=a_{p}^{\alpha}a_{r}^{\beta}a_{\gamma\delta q}^{\gamma}a_{\alpha\beta s}^{\delta}\varepsilon^{pq}\varepsilon^{rs}$ & $J_{23}=a^{\alpha}a^{\beta}a_{\delta}^{\gamma}a_{\mu\gamma p}^{\delta}a_{\beta\alpha q}^{\mu}\varepsilon^{pq}$\tabularnewline
			$J_{16}=a_{p}^{\alpha}a_{r}^{\beta}a_{\beta\delta q}^{\gamma}a_{\alpha\gamma s}^{\delta}\varepsilon^{pq}\varepsilon^{rs}$ & $J_{24}=a^{\alpha}a^{\beta}a_{\mu}^{\gamma}a_{\delta\beta p}^{\delta}a_{\gamma\alpha q}^{\mu}\varepsilon^{pq}$\tabularnewline
			$J_{17}=a_{\gamma}^{\alpha}a_{\delta pk}^{\beta}a_{\beta qr}^{\gamma}a_{\alpha sl}^{\delta}\varepsilon^{pq}\varepsilon^{rs}\varepsilon^{kl}$ & $J_{25}=a^{\alpha}a^{\beta}a_{\gamma\alpha p}^{\gamma}a_{\beta rq}^{\delta}a_{\mu\delta s}^{\mu}\varepsilon^{pq}\varepsilon^{rs}$\tabularnewline
			$J_{18}=a_{\delta}^{\alpha}a_{\gamma\beta p}^{\beta}a_{\alpha rk}^{\gamma}a_{qsl}^{\delta}\varepsilon^{pq}\varepsilon^{rs}\varepsilon^{kl}$ & $J_{26}=a^{\alpha}a^{\beta}a_{\mu\delta\gamma}^{\gamma}a_{\beta pr}^{\delta}a_{\alpha qs}^{\mu}\varepsilon^{pq}\varepsilon^{rs}$\tabularnewline
			$J_{19}=a_{\beta pr}^{\alpha}a_{\delta qs}^{\beta}a_{\alpha km}^{\gamma}a_{\gamma ln}^{\delta}\varepsilon^{pq}\varepsilon^{rs}\varepsilon^{kl}\varepsilon^{mn}$ & $J_{27}=a_{p}^{\alpha}a_{r}^{\beta}a_{\mu\delta k}^{\gamma}a_{\beta\alpha l}^{\delta}a_{\gamma qs}^{\mu}\varepsilon^{pq}\varepsilon^{rs}\varepsilon^{kl}$\tabularnewline
		\end{tabular}
		\par\end{center}
	
	\begin{center}
		\begin{tabular}{l}
			$J_{28}=a_{p}^{\alpha}a_{\mu\delta q}^{\beta}a_{\gamma rm}^{\gamma}a_{\beta sk}^{\delta}a_{\alpha ln}^{\mu}\varepsilon^{pq}\varepsilon^{rs}\varepsilon^{kl}\varepsilon^{mn}$\tabularnewline
		\end{tabular}
		\par\end{center}
	
	\begin{center}
		\begin{tabular}{ll}
			$J_{29}=a^{\alpha}a^{\beta}a^{\gamma}a^{p}a_{\delta}^{q}a_{\alpha\beta\gamma}^{\delta}\varepsilon_{pq}$ & $J_{33}=a^{\alpha}a^{\beta}a_{\nu}^{\gamma}a_{\mu\delta\gamma}^{\delta}a_{\beta pr}^{\mu}a_{\alpha qs}^{\nu}\varepsilon^{pq}\varepsilon^{rs}$\tabularnewline
			$J_{30}=a^{\alpha}a^{\beta}a^{\gamma}a^{\delta}a_{\mu\delta\nu}^{\mu}a_{\alpha\beta\gamma}^{\nu}$ & $J_{34}=a^{\alpha}a^{\beta}a_{\delta}^{\gamma}a_{\nu\mu\gamma}^{\delta}a_{\beta pr}^{\mu}a_{\alpha qs}^{\nu}\varepsilon^{pq}\varepsilon^{rs}$\tabularnewline
			$J_{31}=a^{\alpha}a^{\beta}a^{\gamma}a^{\delta}a_{\gamma\delta\nu}^{\mu}a_{\alpha\beta\mu}^{\nu}$ & $J_{35}=a^{\alpha}a^{\beta}a_{\mu\nu p}^{\gamma}a_{\gamma\delta q}^{\delta}a_{\beta rk}^{\mu}a_{\alpha sl}^{\nu}\varepsilon^{pq}\varepsilon^{rs}\varepsilon^{kl}$\tabularnewline
			$J_{32}=a^{\alpha}a^{\beta}a_{\nu}^{\gamma}a_{\mu}^{\delta}a_{\delta\gamma p}^{\mu}a_{\beta\alpha q}^{\nu}\varepsilon^{pq}$ & $J_{36}=a_{\nu}^{\alpha}a_{r}^{\beta}a_{p}^{\gamma}a_{\beta\gamma\mu}^{\delta}a_{\delta qk}^{\mu}a_{\alpha sl}^{\nu}\varepsilon^{pq}\varepsilon^{rs}\varepsilon^{kl}$\tabularnewline
		\end{tabular}
		\par\end{center}
	
	\begin{center}
		\begin{tabular}{l}
			$J_{37}=a_{\nu}^{\alpha}a_{\delta}^{\beta}a_{\mu pr}^{\gamma}a_{\gamma qs}^{\delta}a_{\beta km}^{\mu}a_{\alpha ln}^{\nu}\varepsilon^{pq}\varepsilon^{rs}\varepsilon^{kl}\varepsilon^{mn}$\tabularnewline
			$J_{38}=a_{p}^{\alpha}a_{\delta\nu q}^{\beta}a_{\mu rk}^{\gamma}a_{\gamma sl}^{\delta}a_{\beta mw}^{\mu}a_{\alpha nv}^{\nu}\varepsilon^{pq}\varepsilon^{rs}\varepsilon^{kl}\varepsilon^{mn}\varepsilon^{wv}$\tabularnewline
			$J_{39}=a_{\nu pk}^{\alpha}a_{\mu qr}^{\beta}a_{\delta sl}^{\gamma}a_{\gamma um}^{\delta}a_{\beta ng}^{\mu}a_{\alpha vh}^{\nu}\varepsilon^{pq}\varepsilon^{rs}\varepsilon^{kl}\varepsilon^{mn}\varepsilon^{uv}\varepsilon^{gh}$\tabularnewline
		\end{tabular}
		\par\end{center}
	
	\begin{center}
		\begin{tabular}{ll}
			$J_{40}=a^{\alpha}a^{\beta}a^{\gamma}a^{\delta}a_{\nu}^{\mu}a_{\delta\mu\eta}^{\nu}a_{\beta\alpha\gamma}^{\eta}$ & $J_{44}=a^{\alpha}a^{\beta}a_{\nu\eta p}^{\gamma}a_{\delta\mu r}^{\delta}a_{\gamma qk}^{\mu}a_{\beta lm}^{\nu}a_{\alpha sn}^{\eta}\varepsilon^{pq}\varepsilon^{rs}\varepsilon^{kl}\varepsilon^{mn}$\tabularnewline
			$J_{41}=a^{\alpha}a^{\beta}a^{\gamma}a^{\delta}a_{\delta\nu\eta}^{\mu}a_{\gamma\mu p}^{\nu}a_{\alpha\beta q}^{\eta}\varepsilon^{pq}$ & $J_{45}=a^{\alpha}a^{\beta}a^{\gamma}a^{\delta}a_{\xi}^{\mu}a_{\delta\mu\eta}^{\nu}a_{\gamma\nu p}^{\eta}a_{\beta\alpha q}^{\xi}\varepsilon^{pq}$\tabularnewline
			$J_{42}=a^{\alpha}a^{\beta}a_{\eta}^{\gamma}a_{\mu}^{\delta}a_{\gamma\delta\nu}^{\mu}a_{\beta pr}^{\nu}a_{\alpha qs}^{\eta}\varepsilon^{pq}\varepsilon^{rs}$ & $J_{46}=a^{\alpha}a^{\beta}a^{\gamma}a^{\delta}a_{\xi\nu\eta}^{\mu}a_{\delta\gamma\mu}^{\nu}a_{\beta pr}^{\eta}a_{\alpha qs}^{\xi}\varepsilon^{pq}\varepsilon^{rs}$\tabularnewline
			$J_{43}=a^{\alpha}a^{\beta}a_{\nu}^{\gamma}a_{\gamma\mu\eta}^{\delta}a_{\delta kp}^{\mu}a_{\beta qr}^{\nu}a_{\alpha sl}^{\eta}\varepsilon^{pq}\varepsilon^{rs}\varepsilon^{kl}$ & $J_{47}=a^{\alpha}a^{\beta}a^{\gamma}a^{\delta}a^{\mu}a^{\nu}a_{\xi\varphi\nu}^{\eta}a_{\delta\mu\eta}^{\xi}a_{\beta\gamma\alpha}^{\varphi}$\tabularnewline
		\end{tabular}
		\par\end{center}
	
	\begin{center}
		\begin{tabular}{lll}
			$K_{1}=a^{p}x^{q}\varepsilon_{pq}$ & $K_{4}=a_{\alpha}^{p}x^{\alpha}x^{q}\varepsilon_{pq}$ & $K_{7}=a^{\alpha}a_{\alpha}^{\beta}a_{\gamma\beta\delta}^{\gamma}x^{\delta}$\tabularnewline
			$K_{2}=a^{\alpha}a_{\alpha}^{p}x^{q}\varepsilon_{pq}$ & $K_{5}=a_{\alpha\beta\gamma}^{\alpha}x^{\beta}x^{\gamma}$ & $K_{8}=a^{\alpha}a_{\delta\alpha p}^{\beta}a_{\gamma\beta q}^{\gamma}x^{\delta}\varepsilon^{pq}$ \tabularnewline
			$K_{3}=a^{\alpha}a_{\beta\alpha\gamma}^{\beta}x^{\gamma}$ & $K_{6}=a^{\alpha}a_{\gamma}^{\beta}a_{\alpha\beta\delta}^{\gamma}x^{\delta}$  & $K_{9}=a_{\beta}^{\alpha}a_{\alpha\gamma\delta}^{\beta}x^{\gamma}x^{\delta}$\tabularnewline
		\end{tabular}
		\par\end{center}
	
	\begin{center}
		\begin{tabular}{ll}
			$K_{10}=a_{\gamma}^{\alpha}a_{\beta\alpha\delta}^{\beta}x^{\gamma}x^{\delta}$ & $K_{15}=a^{\alpha}a_{\delta}^{\beta}a_{\gamma\beta p}^{\gamma}a_{\alpha\mu q}^{\delta}x^{\mu}\varepsilon^{pq}$\tabularnewline
			$K_{11}=a_{\gamma\delta p}^{\alpha}a_{\beta\alpha q}^{\beta}x^{\gamma}x^{\delta}\varepsilon^{pq}$ & $K_{16}=a^{\alpha}a_{\gamma}^{\beta}a_{\delta\beta p}^{\gamma}a_{\alpha\mu q}^{\delta}x^{\mu}\varepsilon^{pq}$\tabularnewline
			$K_{12}=a^{\alpha}a^{\beta}a^{\gamma}a_{\alpha\beta\gamma}^{q}x^{p}\varepsilon_{pq}$ & $K_{17}=a^{\alpha}a_{\delta}^{\beta}a_{\gamma\alpha p}^{\gamma}a_{\beta\mu q}^{\delta}x^{\mu}\varepsilon^{pq}$\tabularnewline
			$K_{13}=a^{\alpha}a^{\beta}a_{\alpha\beta\gamma}^{q}x^{\gamma}x^{p}\varepsilon_{pq}$ & $K_{18}=a^{\alpha}a_{\beta\alpha p}^{\beta}a_{\mu rq}^{\gamma}a_{\delta\gamma s}^{\delta}x^{\mu}\varepsilon^{pq}\varepsilon^{rs}$\tabularnewline
			$K_{14}=a^{\alpha}a_{\delta}^{\beta}a_{\alpha}^{\gamma}a_{\beta\gamma\mu}^{\delta}x^{\mu}$ & $K_{19}=a^{\alpha}a_{\beta\delta\gamma}^{\beta}a_{\mu pr}^{\gamma}a_{\alpha qs}^{\delta}x^{\mu}\varepsilon^{pq}\varepsilon^{rs}$\tabularnewline
		\end{tabular}
		\par\end{center}
	
	\begin{center}
		\begin{tabular}{ll}
			$K_{20}=a^{\alpha}a_{\alpha\beta\gamma}^{q}x^{p}x^{\beta}x^{\gamma}\varepsilon_{pq}$ & $K_{27}=a_{\alpha\beta\gamma}^{q}x^{p}x^{\alpha}x^{\beta}x^{\gamma}\varepsilon_{pq}$\tabularnewline
			$K_{21}=a_{\gamma}^{\alpha}a_{\delta}^{\beta}a_{\alpha\beta\mu}^{\gamma}x^{\delta}x^{\mu}$ & $K_{28}=a^{\alpha}a^{\beta}a^{p}a_{\gamma}^{q}a_{\beta\delta\alpha}^{\gamma}x^{\delta}\varepsilon_{pq}$\tabularnewline
			$K_{22}=a_{\gamma}^{\alpha}a_{\beta\alpha p}^{\beta}a_{\delta\mu q}^{\gamma}x^{\delta}x^{\mu}\varepsilon^{pq}$  & $K_{29}=a^{\alpha}a^{\beta}a^{\gamma}a_{\delta\mu\nu}^{\delta}a_{\alpha\beta\gamma}^{\mu}x^{\nu}$\tabularnewline
			$K_{23}=a_{\beta}^{\alpha}a_{\gamma\alpha p}^{\beta}a_{\delta\mu q}^{\gamma}x^{\delta}x^{\mu}\varepsilon^{pq}$ & $K_{30}=a^{\alpha}a^{\beta}a^{\gamma}a_{\gamma\mu\nu}^{\delta}a_{\alpha\beta\delta}^{\mu}x^{\nu}$ \tabularnewline
			$K_{24}=a_{\gamma}^{\alpha}a_{\beta\mu p}^{\beta}a_{\alpha\delta q}^{\gamma}x^{\delta}x^{\mu}\varepsilon^{pq}$  & $K_{31}=a^{\alpha}a^{p}a_{\beta}^{q}a_{\delta\gamma\alpha}^{\beta}x^{\gamma}x^{\delta}\varepsilon_{pq}$\tabularnewline
			$K_{25}=a_{\alpha\delta p}^{\alpha}a_{\mu rq}^{\beta}a_{\gamma\beta s}^{\gamma}x^{\delta}x^{\mu}\varepsilon^{pq}\varepsilon^{rs}$ & $K_{32}=a^{\alpha}a^{\beta}a_{\delta\mu\nu}^{\gamma}a_{\alpha\beta\gamma}^{\delta}x^{\mu}x^{\nu}$\tabularnewline
			$K_{26}=a_{\alpha\beta\gamma}^{\alpha}a_{\mu pr}^{\beta}a_{\delta qs}^{\gamma}x^{\delta}x^{\mu}\varepsilon^{pq}\varepsilon^{rs}$ & $K_{33}=a^{\alpha}a^{\beta}a_{\gamma\delta\nu}^{\gamma}a_{\alpha\beta\mu}^{\delta}x^{\mu}x^{\nu}$\tabularnewline
		\end{tabular}
		\par\end{center}
	
	\begin{center}
		\begin{tabular}{ll}
			$K_{34}=a^{\alpha}a_{\mu}^{\beta}a_{\delta}^{\gamma}a_{\gamma\beta p}^{\delta}a_{\alpha\nu q}^{\mu}x^{\nu}\varepsilon^{pq}$ & $K_{39}=a^{\alpha}a_{\beta\alpha\gamma}^{\beta}a_{\delta\mu\nu}^{\gamma}x^{\delta}x^{\mu}x^{\nu}$\tabularnewline
			$K_{35}=a^{\alpha}a_{\mu}^{\beta}a_{\gamma\beta\delta}^{\gamma}a_{\alpha pr}^{\delta}a_{\nu qs}^{\mu}x^{\nu}\varepsilon^{pq}\varepsilon^{rs}$ & $K_{40}=a^{\alpha}a_{\delta\alpha\gamma}^{\beta}a_{\beta\mu\nu}^{\gamma}x^{\delta}x^{\mu}x^{\nu}$\tabularnewline
			$K_{36}=a^{\alpha}a_{\gamma}^{\beta}a_{\mu\delta\beta}^{\gamma}a_{\alpha pr}^{\delta}a_{\nu qs}^{\mu}x^{\nu}\varepsilon^{pq}\varepsilon^{rs}$ & $K_{41}=a_{\delta}^{\alpha}a_{\gamma}^{\beta}a_{\alpha\beta p}^{\gamma}a_{\mu\nu q}^{\delta}x^{\mu}x^{\nu}\varepsilon^{pq}$\tabularnewline
			$K_{37}=a^{p}a_{\alpha}^{q}a_{\beta\delta\gamma}^{\alpha}x^{\beta}x^{\gamma}x^{\delta}\varepsilon_{pq}$ & $K_{42}=a_{\delta}^{\alpha}a_{\alpha\beta\gamma}^{\beta}a_{\mu pr}^{\gamma}a_{\nu qs}^{\delta}x^{\mu}x^{\nu}\varepsilon^{pq}\varepsilon^{rs}$ \tabularnewline
			$K_{38}=a^{\alpha}a_{\delta\mu p}^{\beta}a_{\gamma\beta q}^{\gamma}a_{\alpha rk}^{\delta}a_{\nu sl}^{\mu}x^{\nu}\varepsilon^{pq}\varepsilon^{rs}\varepsilon^{kl}$ & $K_{43}=a_{\beta}^{\alpha}a_{\alpha\delta\gamma}^{\beta}a_{\mu pr}^{\gamma}a_{\nu qs}^{\delta}x^{\mu}x^{\nu}\varepsilon^{pq}\varepsilon^{rs}$ \tabularnewline
		\end{tabular}
		\par\end{center}
	
	\begin{center}
		\begin{tabular}{cc}
			$K_{44}=a_{\alpha}^{q}a_{\beta\gamma\delta}^{\alpha}x^{p}x^{\beta}x^{\gamma}x^{\delta}\varepsilon_{pq}$ & $K_{45}=a_{\gamma\delta p}^{\alpha}a_{\beta\alpha q}^{\beta}a_{\mu rk}^{\gamma}a_{\nu sl}^{\delta}x^{\mu}x^{\nu}\varepsilon^{pq}\varepsilon^{rs}\varepsilon^{kl}$\tabularnewline
		\end{tabular}
		\par\end{center}
	
	\begin{center}
		\begin{tabular}{ll}
			$K_{46}=a_{\alpha\beta\gamma}^{\alpha}a_{\delta\mu\nu}^{\beta}x^{\gamma}x^{\delta}x^{\mu}x^{\nu}$ & $K_{50}=a^{\alpha}a^{\beta}a_{\delta}^{\gamma}a_{\beta\gamma\mu}^{\delta}a_{\alpha\nu\eta}^{\mu}x^{\nu}x^{\eta}$\tabularnewline
			$K_{47}=a_{\beta\gamma\delta}^{\alpha}a_{\alpha\mu\nu}^{\beta}x^{\gamma}x^{\delta}x^{\mu}x^{\nu}$ & $K_{51}=a^{\alpha}a^{\beta}a_{\delta\mu\nu}^{\gamma}a_{\gamma\eta p}^{\delta}a_{\alpha\beta q}^{\mu}x^{\nu}x^{\eta}\varepsilon^{pq}$\tabularnewline
			$K_{48}=a^{\alpha}a^{\beta}a^{\gamma}a_{\mu}^{\delta}a_{\gamma\delta\nu}^{\mu}a_{\alpha\beta\eta}^{\nu}x^{\eta}$ & $K_{52}=a^{\alpha}a_{\nu}^{\beta}a_{\delta}^{\gamma}a_{\beta\gamma\mu}^{\delta}a_{\alpha pr}^{\mu}a_{\eta qs}^{\nu}x^{\eta}\varepsilon^{pq}\varepsilon^{rs}$\tabularnewline
			$K_{49}=a^{\alpha}a^{\beta}a^{\gamma}a_{\mu\nu\eta}^{\delta}a_{\gamma\delta p}^{\mu}a_{\alpha\beta q}^{\nu}x^{\eta}\varepsilon^{pq}$ & $K_{53}=a^{\alpha}a_{\mu}^{\beta}a_{\beta\delta\nu}^{\gamma}a_{\gamma kp}^{\delta}a_{\alpha qr}^{\mu}a_{\eta sl}^{\nu}x^{\eta}\varepsilon^{pq}\varepsilon^{rs}\varepsilon^{kl}$\tabularnewline
		\end{tabular}
		\par\end{center}
	
	\begin{center}
		\begin{tabular}{cc}
			$K_{54}=a^{\alpha}a_{\gamma}^{\beta}a_{\alpha\beta\delta}^{\gamma}a_{\mu\eta\nu}^{\delta}x^{\mu}x^{\nu}x^{\eta}$ & $K_{55}=a^{\alpha}a_{\mu\nu p}^{\beta}a_{\gamma\delta r}^{\gamma}a_{\beta qk}^{\delta}a_{\alpha lm}^{\mu}a_{\eta sn}^{\nu}x^{\eta}\varepsilon^{pq}\varepsilon^{rs}\varepsilon^{kl}\varepsilon^{mn}$\tabularnewline
		\end{tabular}
		\par\end{center}
	
	\begin{center}
		\begin{tabular}{l}
			$K_{56}=a^{\alpha}a_{\alpha\gamma\delta}^{\beta}a_{\eta\beta p}^{\gamma}a_{\mu\nu q}^{\delta}x^{\mu}x^{\nu}x^{\eta}\varepsilon^{pq}$\tabularnewline
			$K_{57}=a_{\mu}^{\alpha}a_{\gamma}^{\beta}a_{\alpha\beta\delta}^{\gamma}a_{\eta pr}^{\delta}a_{\nu qs}^{\mu}x^{\nu}x^{\eta}\varepsilon^{pq}\varepsilon^{rs}$\tabularnewline
			$K_{58}=a_{\delta}^{\alpha}a_{\alpha\gamma\mu}^{\beta}a_{\beta kp}^{\gamma}a_{\eta qr}^{\delta}a_{\nu sl}^{\mu}x^{\nu}x^{\eta}\varepsilon^{pq}\varepsilon^{rs}\varepsilon^{kl}$\tabularnewline
			$K_{59}=a_{\beta}^{\alpha}a_{\alpha\gamma\delta}^{\beta}a_{\mu\eta\nu}^{\gamma}x^{\delta}x^{\mu}x^{\nu}x^{\eta}$\tabularnewline
			$K_{60}=a_{\delta\mu p}^{\alpha}a_{\beta\gamma r}^{\beta}a_{\alpha qk}^{\gamma}a_{\nu lm}^{\delta}a_{\eta sn}^{\mu}x^{\nu}x^{\eta}\varepsilon^{pq}\varepsilon^{rs}\varepsilon^{kl}\varepsilon^{mn}$\tabularnewline
			$K_{61}=a_{\beta\gamma\nu}^{\alpha}a_{\alpha\eta p}^{\beta}a_{\delta\mu q}^{\gamma}x^{\delta}x^{\mu}x^{\nu}x^{\eta}\varepsilon^{pq}$\tabularnewline
		\end{tabular}
		\par\end{center}
	
	\begin{center}
		\begin{tabular}{ll}
			$K_{62}=a^{\alpha}a^{\beta}a^{\gamma}a_{\eta}^{\delta}a_{\nu\delta\gamma}^{\mu}a_{\mu\beta p}^{\nu}a_{\alpha\xi q}^{\eta}x^{\xi}\varepsilon^{pq}$ & $K_{69}=a_{\beta\gamma\delta}^{\alpha}a_{\nu\alpha\mu}^{\beta}a_{\xi pr}^{\gamma}a_{\eta qs}^{\delta}x^{\mu}x^{\nu}x^{\eta}x^{\xi}\varepsilon^{pq}\varepsilon^{rs}$ \tabularnewline
			$K_{63}=a^{\alpha}a^{\beta}a^{\gamma}a_{\mu\nu\eta}^{\delta}a_{\xi\gamma\delta}^{\mu}a_{\beta pr}^{\nu}a_{\alpha rs}^{\eta}x^{\xi}\varepsilon^{pq}\varepsilon^{rs}$  & $K_{70}=a^{\alpha}a^{\beta}a^{\gamma}a^{\delta}a^{\mu}a_{\xi\varphi\eta}^{\nu}a_{\delta\nu\mu}^{\eta}a_{\beta\gamma\alpha}^{\xi}x^{\varphi}$\tabularnewline
			$K_{64}=a^{\alpha}a^{\beta}a_{\nu}^{\gamma}a_{\mu\gamma\beta}^{\delta}a_{\delta\alpha p}^{\mu}a_{\eta\xi q}^{\nu}x^{\eta}x^{\xi}\varepsilon^{pq}$ & $K_{71}=a^{\alpha}a^{\beta}a^{\gamma}a^{\delta}a_{\xi\nu\eta}^{\mu}a_{\delta\mu\varphi}^{\nu}a_{\beta\gamma\alpha}^{\eta}x^{\xi}x^{\varphi}$\tabularnewline
			$K_{65}=a^{\alpha}a^{\beta}a_{\delta\mu\nu}^{\gamma}a_{\xi\gamma\eta}^{\delta}a_{\beta pr}^{\mu}a_{\alpha qs}^{\nu}x^{\eta}x^{\xi}\varepsilon^{pq}\varepsilon^{rs}$  & $K_{72}=a^{\alpha}a^{\beta}a^{\gamma}a_{\xi\mu\nu}^{\delta}a_{\delta\eta\varphi}^{\mu}a_{\beta\gamma\alpha}^{\nu}x^{\eta}x^{\xi}x^{\varphi}$\tabularnewline
			$K_{66}=a^{\alpha}a_{\mu}^{\beta}a_{\delta\beta\alpha}^{\gamma}a_{\gamma\xi p}^{\delta}a_{\nu\eta q}^{\mu}x^{\nu}x^{\eta}x^{\xi}\varepsilon^{pq}$ & $K_{73}=a^{\alpha}a^{\beta}a_{\delta\xi\mu}^{\gamma}a_{\gamma\eta\varphi}^{\delta}a_{\beta\alpha\nu}^{\mu}x^{\nu}x^{\eta}x^{\xi}x^{\varphi}$\tabularnewline
			$K_{67}=a^{\alpha}a_{\gamma\delta\mu}^{\beta}a_{\eta\beta\nu}^{\gamma}a_{\alpha pr}^{\delta}a_{\xi qs}^{\mu}x^{\nu}x^{\eta}x^{\xi}\varepsilon^{pq}\varepsilon^{rs}$  & $K_{74}=a^{\alpha}a_{\delta\xi\gamma}^{\beta}a_{\beta\eta\varphi}^{\gamma}a_{\alpha\mu\nu}^{\delta}x^{\mu}x^{\nu}x^{\eta}x^{\xi}x^{\varphi}$\tabularnewline
			$K_{68}=a_{\delta}^{\alpha}a_{\gamma\alpha\xi}^{\beta}a_{\beta\eta p}^{\gamma}a_{\mu\nu q}^{\delta}x^{\mu}x^{\nu}x^{\eta}x^{\xi}\varepsilon^{pq}$ & $K_{75}=a_{\beta\delta\gamma}^{\alpha}a_{\alpha\mu\nu}^{\beta}a_{\xi\eta\varphi}^{\gamma}x^{\delta}x^{\mu}x^{\nu}x^{\eta}x^{\xi}x^{\varphi}$\tabularnewline
		\end{tabular}
		\par\end{center}

	and $\varepsilon ^{pq}=\varepsilon _{pq}=q-p$. 
	
\end{theorem}

\section{Syzygies between centro-affine covariants of  $\mathcal{S}(2,\mathbb{R},\left \{ 0,1,3\right \} )$}

$S$ is said to be a syzygy between the elements of  $\mathcal{S}(2,\mathbb{R},\left \{ 0,1,3\right \} )$ if and only if $S\equiv 0$ is an identity with respect to the variables from $a\in \mathcal{C}(n,\Bbbk ,\Omega )$ and the contravariant vector $x$, but not an identity with respect to the elements from $\mathcal{C}(n,\Bbbk ,\Omega )$.

A minimal system of $27$ syzygies relating elements of the minimal system of
(\ref{system}) where $\Omega =\left \{ 0,1,2\right \} $ is found in \cite[Theorem 17.1]{Sibirskii} and \cite{DanilyukandSibirskii}.

The action of the general linear group $GL(2)$ on $\mathbb{R}^{2}:(Q,x)\mapsto Qx$ induces a representation 
\begin{equation*}
\begin{tabular}{llll}
$\rho :$ & $GL(2)$ & $\rightarrow $ & $GL(\mathcal{C}(\mathbb{R},\Omega ))$
\\ 
& $Q$ & $\mapsto $ & $\rho (Q)$
\end{tabular}
\  \  \  \ 
\end{equation*}
defined by

\begin{equation}
\begin{tabular}{lll}
$\rho (Q)a^{j}$ & $=$ & $\displaystyle \sum_{i=1}^{2}Q_{i}^{j}a^{i}$ \\ 

$\rho (Q)a_{\alpha _{1}}^{j}$& $=$ & $\displaystyle
\sum _{i=1}^{2}\sum _{\alpha _{1}=1}^{2}\sum _{\beta _{1}=1}^{2}Q_{i}^{j}P_{\alpha _{1}}^{\beta _{1}}a_{\beta _{1}}^{i}$\\

$\rho (Q)a_{\alpha _{1}\alpha _{2}\alpha _{3}}^{j}$ & $=$ & $\displaystyle
\sum _{i=1}^{2}\sum _{\alpha _{1}=1}^{2}\sum _{\alpha _{2}=1}^{2}\sum
_{\alpha _{3}=1}^{2}\sum _{\beta _{1}=1}^{2}\sum _{\beta _{2}=1}^{2}\sum
_{\beta _{3}=1}^{2}Q_{i}^{j}P_{\alpha _{1}}^{\beta _{1}}P_{\alpha
	_{2}}^{\beta _{2}}P_{\alpha _{3}}^{\beta _{3}}a_{\beta _{1}\beta _{2}\beta
	_{3}}^{i}$
\end{tabular}
\ ,j=1,2  \label{laws}
\end{equation}

where $Q$ is a matrix of $GL(2)$ and $P$ its inverse. The formula (\ref{laws}) is called the formula of the centro-affine transformations.

\begin{theorem}
	If $J_{7}\neq 0$, then the system (\ref{syscubic}) can be transformed into
	the following form 
	$$\left \{ 
	\begin{array}{l}
	\frac{dy^{1}}{dt}=b^{1}+b^{1}_{1}y^{1}+b^{1}_{2}y^{2}+b_{111}^{1}(y^{1})^{3}+3b_{112}^{1}(y^{1})^{2}y^{2}+3b_{122}^{1}y^{1}(y^{2})^{2}+b_{222}^{1}(y^{2})^{3} \\ 
	\frac{dy^{2}}{dt}=b^{2}+b^{2}_{1}y^{1}+b^{2}_{2}y^{2}+b_{111}^{2}(y^{1})^{3}+3b_{112}^{2}(y^{1})^{2}y^{2}+3b_{122}^{2}y^{1}(y^{2})^{2}+b_{222}^{2}(y^{2})^{3}
	\end{array}
	\right. .$$
	where
	$b^{1}=0$, $b^{2}=1$, $b^{1}_{1}=J_{1}J_{7}-J_{13}$, $b^{1}_{2}=J_{6}$, $b^{2}_{1}=\frac{1}{J_{7}^{2}}(J_{3}J_{7}+\frac{1}{2}J_{4}J_{6})$, $b^{2}_{2}=\frac{J_{13}}{J_{7}}$,\\ $b_{111}^{1}=\frac{1}{J_{7}^{3}}(J_{7}J_{25}+\frac{1}{2}J_{4}J_{30})$, $b_{112}^{1}=\frac{1}{J_{7}^{2}}(\frac{1}{2}J_{4}J_{20}-J_{7}J_{14})$, $b_{122}^{1}=\frac{1}{J_{7}}(J_{7}^{2}-J_{30})$, $b_{222}^{1}=-J_{20} $,\\ $b_{111}^{2}=\frac{1}{J_{7}^{4}}\left( \frac{1}{4}
	J_{4}^{2}J_{20}-J_{7}(J_{4}J_{14}+J_{7}(J_{10}+J_{11}))\right) $,  $b_{112}^{2}=\frac{1}{J_{7}^{3}}\left( \frac{1}{2}
	J_{4}(J_{7}^{2}-J_{30})-J_{7}J_{25})\right)$, \\ $b_{122}^{2}=\frac{1}{J_{7}^{2}}\left( J_{7}J_{14}-\frac{1}{2}J_{4}J_{20}\right)$, $b_{222}^{2}=\frac{J_{30}}{J_{7}}$.

\end{theorem}

\proof

Since $J_{7}\neq 0$, we can consider the following matrix

\begin{equation}
q=
\begin{pmatrix}
a^{2} & -a^{1} \\ 
\frac{1}{J_{7}}a^{\alpha }a_{\beta \alpha 1}^{\beta } & \frac{1}{J_{7}}
a^{\alpha }a_{\beta \alpha 2}^{\beta }
\end{pmatrix}.
\label{varch}
\end{equation}

The matrix $q$ is invertible since $\det q=1$. Then, consider the centro-affine transformation

\begin{equation*}
y^{j}=q_{i}^{j}x^{i},i,j=1,2
\end{equation*}

i.e.

\begin{equation*}
y^{1}=-K_{1},\, \,y^{2}=\frac{K_{3}}{J_{7}}
\end{equation*}

The system (\ref{syscubic}) can be transformed into the system
\begin{equation*}
\frac{dy^{j}}{dt}=b^{j}+b^{j}_{\alpha}y^{\alpha }+b_{\alpha \beta \gamma }^{j}y^{\alpha }y^{\beta}y^{\gamma },j,\alpha ,\beta ,\gamma=1,2
\end{equation*}

where for $j,\alpha ,\beta ,\gamma =1,2,$
\begin{equation*}
b^{j}=q_{\alpha }^{j}a^{\alpha },\; b_{\alpha}^{j}=q_{i}^{j}p_{\alpha }^{j_{1}}a_{j_{1}}^{i} \text{ and }b_{\alpha \beta \gamma }^{j}=q_{i}^{j}p_{\alpha }^{j_{1}}p_{\beta }^{j_{2}}p_{\gamma}^{j_{3}}a_{j_{1}j_{2}j_{3}}^{i},i,j_{1},j_{2},j_{3}=1,2
\end{equation*}
and $p$ is the inverse of $q.$ Hence,

\begin{itemize}[leftmargin=1cm] 
	\item[$b^{1}$] $=q_{i}^{1}a^{i}=0$

	\item[$b^{2}$] $=q_{i}^{2}a^{i}=1$

	\item[$b^{1}_{1}$] $=q_{i}^{1}p_{1}^{j_{1}}a_{j_{1}}^{i}=\frac{1}{J_{7}}(a^{\mu }\varepsilon _{i\mu })(a^{\alpha }a_{\beta
		\alpha i}^{\beta }\varepsilon ^{j_{1}i})(a_{j_{1}}^{i})=\frac{1}{J_{7}} a^{\alpha}a^{q}a^{p}_{r}a^{\beta}_{\beta \alpha s}\varepsilon _{pq}\varepsilon ^{rs} =\frac{1}{J_{7}}(J_{1}J_{7}-J_{13}) $

	\item[$b^{1}_{2}$] $=q_{i}^{1}p_{2 }^{j_{1}}a_{j_{1}}^{i}=(a^{\mu }\varepsilon _{i\mu })(a^{j_{1}})(a_{j_{1}}^{i})=a^{\alpha}a^{q}a^{p}_{\alpha}\varepsilon _{pq} =J_{6}$
	
	\item[$b^{2}_{1}$] $=q_{i}^{2}p_{1 }^{j_{1}}a_{j_{1}}^{i}=\frac{1}{J_{7}^{2}}(a^{\alpha }a_{\beta \alpha i}^{\beta
	})(a^{\alpha }a_{\beta
		\alpha i}^{\beta }\varepsilon ^{j_{1}i})(a_{j_{1}}^{i})= \frac{1}{J_{7}^{2}}a^{\alpha}a^{\beta}a^{\gamma}_{p}a^{\delta}_{\delta \gamma \alpha}a^{\mu}_{\mu \beta q}\varepsilon ^{pq}=\frac{1}{J_{7}^{2}}(J_{3}J_{7}+\frac{1}{2}J_{4}J_{6}) $

	\item[$b^{2}_{2}$] $=q_{i}^{2}p_{2 }^{j_{1}}a_{j_{1}}^{i}=\frac{1}{J_{7}}(a^{\alpha }a_{\beta \alpha i}^{\beta
	})(a^{j_{1}})(a_{j_{1}}^{i})=\frac{1}{J_{7}}a^{\alpha}a^{\beta}a^{\gamma}_{\beta}a^{\delta}_{\delta \gamma \alpha}=\frac{J_{13}}{J_{7}}$

	\item[$b_{111}^{1}$] $
	=q_{i}^{1}p_{1}^{j_{1}}p_{1}^{j_{2}}p_{1}^{j_{3}}a_{j_{1}j_{2}j_{3}}^{i}$
	\newline
	$=\frac{1}{J_{7}^{3}}(a^{\mu }\varepsilon _{i\mu })(a^{\alpha }a_{\beta
		\alpha i}^{\beta }\varepsilon ^{j_{1}i})(a^{\alpha }a_{\beta \alpha
		i}^{\beta }\varepsilon ^{j_{2}i})(a^{\alpha }a_{\beta \alpha i}^{\beta
	}\varepsilon ^{j_{3}i})a_{j_{1}j_{2}j_{3}}^{i}$\newline
	$=\frac{1}{J_{7}^{3}}(a^{\alpha }a^{\beta }a^{\gamma }a^{q}a_{\delta \alpha
		s}^{\delta }a_{\mu \beta l}^{\mu }a_{\nu \gamma n}^{\nu
	}a_{rkm}^{p}\varepsilon _{pq}\varepsilon ^{rs}\varepsilon ^{kl}\varepsilon
	^{mn})$\newline
	$=\frac{1}{J_{7}^{3}}(J_{7}J_{25}+\frac{1}{2}J_{4}J_{30}) $
	
	\item[$b_{112}^{1}$] $
	=q_{i}^{1}p_{2}^{j_{1}}p_{1}^{j_{2}}p_{1}^{j_{3}}a_{j_{1}j_{2}j_{3}}^{i}$
	\newline
	$=\frac{1}{J_{7}^{2}}(a^{\mu }\varepsilon _{i\mu })(a^{j_{1}})(a^{\alpha
	}a_{\beta \alpha i}^{\beta }\varepsilon ^{j_{2}i})(a^{\alpha }a_{\beta
		\alpha i}^{\beta }\varepsilon ^{j_{3}i})a_{j_{1}j_{2}j_{3}}^{i}$\newline
	$=\frac{1}{J_{7}^{2}}(a^{\alpha }a^{\beta }a^{\gamma }a^{q}a_{\delta \beta
		s}^{\delta }a_{\mu \gamma l}^{\mu }a_{\alpha rk}^{p}\varepsilon
	_{pq}\varepsilon ^{rs}\varepsilon ^{kl})$\newline
	$=\frac{1}{J_{7}^{2}}(\frac{1}{2}J_{4}J_{20}-J_{7}J_{14})$
	
	\item[$b_{122}^{1}$] $
	=q_{i}^{1}p_{2}^{j_{1}}p_{2}^{j_{2}}p_{1}^{j_{3}}a_{j_{1}j_{2}j_{3}}^{i}$%
	\newline
	$=\frac{1}{J_{7}}(a^{\mu }\varepsilon _{i\mu
	})(a^{j_{1}})(a^{j_{2}})(a^{\alpha }a_{\beta \alpha i}^{\beta }\varepsilon
	^{j_{3}i})a_{j_{1}j_{2}j_{3}}^{i}$\newline
	$=\frac{1}{J_{7}}(a^{\alpha }a^{\beta }a^{\gamma }a^{q}a_{\delta \gamma
		s}^{\delta }a_{\alpha \beta r}^{p}\varepsilon _{pq}\varepsilon ^{rs})$
	\newline
	$=\frac{1}{J_{7}}(J_{7}^{2}-J_{30})$	
	\item[$b_{222}^{1}$] $
	=q_{i}^{1}p_{2}^{j_{1}}p_{2}^{j_{2}}p_{2}^{j_{3}}a_{j_{1}j_{2}j_{3}}^{i}$
	\newline
	$=(a^{\mu }\varepsilon _{i\mu
	})(a^{j_{1}})(a^{j_{2}})(a^{j_{3}})a_{j_{1}j_{2}j_{3}}^{i}$\newline
	$=a^{\alpha }a^{\beta }a^{\gamma }a^{q}a_{\alpha \beta \gamma
	}^{p}\varepsilon _{pq}$\newline
	$=-J_{20}$
	\item[$b_{111}^{2}$] $
	=q_{i}^{2}p_{1}^{j_{1}}p_{1}^{j_{2}}p_{1}^{j_{3}}a_{j_{1}j_{2}j_{3}}^{i}$%
	\newline
	$=\frac{1}{J_{7}^{4}}(a^{\alpha }a_{\beta \alpha i}^{\beta })(a^{\alpha
	}a_{\beta \alpha i}^{\beta }\varepsilon ^{j_{1}i})(a^{\alpha }a_{\beta
		\alpha i}^{\beta }\varepsilon ^{j_{2}i})(a^{\alpha }a_{\beta \alpha
		i}^{\beta }\varepsilon ^{j_{3}i})a_{j_{1}j_{2}j_{3}}^{i}$\newline
	$=\frac{1}{J_{7}^{4}}(a^{\alpha }a^{\beta }a^{\gamma }a^{\delta }a_{\mu
		\alpha \varphi }^{\mu }a_{\nu \beta q}^{\nu }a_{\eta \gamma s}^{\eta
	}a_{\lambda \delta l}^{\lambda }a_{prk}^{\varphi }\varepsilon
	^{pq}\varepsilon ^{rs}\varepsilon ^{kl})$\newline
	$=\frac{1}{J_{7}^{4}}\left( \frac{1}{4}
	J_{4}^{2}J_{20}-J_{7}(J_{4}J_{14}+J_{7}(J_{10}+J_{11}))\right) $	
	\item[$b_{112}^{2}$] $
	=q_{i}^{2}p_{2}^{j_{1}}p_{1}^{j_{2}}p_{1}^{j_{3}}a_{j_{1}j_{2}j_{3}}^{i}$%
	\newline
	$=\frac{1}{J_{7}^{3}}(a^{\alpha }a_{\beta \alpha i}^{\beta
	})(a^{j_{1}})(a^{\alpha }a_{\beta \alpha i}^{\beta }\varepsilon
	^{j_{2}i})(a^{\alpha }a_{\beta \alpha i}^{\beta }\varepsilon
	^{j_{3}i})a_{j_{1}j_{2}j_{3}}^{i}$\newline
	$=\frac{1}{J_{7}^{3}}(a^{\alpha }a^{\beta }a^{\gamma }a^{\delta }a_{\mu
		\gamma \lambda }^{\mu }a_{\nu \alpha q}^{\nu }a_{\eta \beta s}^{\eta
	}a_{\delta pr}^{\lambda }\varepsilon ^{pq}\varepsilon ^{rs})$\newline
	$=\frac{1}{J_{7}^{3}}\left( \frac{1}{2}
	J_{4}(J_{7}^{2}-J_{30})-J_{7}J_{25})\right) $	
	\item[$b_{122}^{2}$] $
	=q_{i}^{2}p_{2}^{j_{1}}p_{2}^{j_{2}}p_{1}^{j_{3}}a_{j_{1}j_{2}j_{3}}^{i}$
	\newline
	$=\frac{1}{J_{7}^{2}}(a^{\alpha }a_{\beta \alpha i}^{\beta
	})(a^{j_{1}})(a^{j_{2}})(a^{\alpha }a_{\beta \alpha i}^{\beta }\varepsilon
	^{j_{3}i})a_{j_{1}j_{2}j_{3}}^{i}$\newline
	$=\frac{1}{J_{7}^{2}}(a^{\alpha }a^{\beta }a^{\gamma }a^{\delta }a_{\mu
		\delta \eta }^{\mu }a_{\nu \alpha q}^{\nu }a_{\beta \gamma p}^{\eta
	}\varepsilon ^{pq})$\newline
	$=\frac{1}{J_{7}^{2}}\left( J_{7}J_{14}-\frac{1}{2}
	J_{4}J_{20}\right) $
	\item[$b_{222}^{2}$] $
	=q_{i}^{2}p_{2}^{j_{1}}p_{2}^{j_{2}}p_{2}^{j_{3}}a_{j_{1}j_{2}j_{3}}^{i}$
	\newline
	$=\frac{1}{J_{7}}(a^{\alpha }a_{\beta \alpha i}^{\beta
	})(a^{j_{1}})(a^{j_{2}})(a^{j_{3}})a_{j_{1}j_{2}j_{3}}^{i}$\newline
	$=\frac{1}{J_{7}}(a^{\alpha }a^{\beta }a^{\gamma }a^{\delta }a_{\mu \delta
		\nu }^{\mu }a_{\alpha \beta \gamma }^{\nu })$\newline
	$=\frac{J_{30}}{J_{7}}$.
\end{itemize}

\begin{theorem} 
	Ther exist $109$ syzygies between $J_{1},\ldots ,J_{47},K_{1},\ldots ,K_{75}.$ where

	\begin{itemize}[leftmargin=1cm]

		\item[$S_{J_{2}}$ :]

		$J_{7}^{2}J_{2}=J_{7}(J_{1}^{2}J_{7}-2J_{1}J_{13}+2J_{3}J_{6})+J_{4}J_{6}^{2}+2J_{13}^{2}$.

		\item[$S_{J_{5}}$ :]

		$J_{7}^{2}J_{5}=2(J_{4}(J_{31}-J_{30})+J_{7}(2J_{26}+3J_{25})+(J_{11}+J_{10})J_{20}+J_{14}^{2})$.
		
		\item[$S_{J_{8}}$ :]
		
		$6J_{7}^{2}J_{8}=3((J_{1}^{2}-J_{2})J_{7}+2(J_{21}-J_{1}J_{12}))J_{14}
		+6(J_{7}J_{12}-J_{40})+(2J_{6}J_{10}+6(J_{34}-J_{33}))J_{6}\\
		+6(J_{8}J_{30}+(J_{13}-J_{12})J_{22})$.

		\item[$S_{J_{9}}$ :]
		
		$2J_{7}^{2}J_{9}=[J_{1}(J_{5}J_{7}+2J_{25})-2(J_{3}J_{14}+(J_{11}+J_{10}))]J_{7}-2(J_{4}J_{6}J_{14}+2J_{13}J_{25})$.

		\item[$S_{J_{10}}$ :]
		
		$2J_{7}^{2}J_{10}=3(2(2J_{26}+3J_{25}-J_{5}J_{7})J_{14}+2(J_{11}
		+J_{10})(J_{30}-J_{31})+(2J_{19}-J_{5}^{2})J_{20})$.
		
		\item[$S_{J_{12}}$ :]

		$J_{7}^{2}J_{12}=J_{1}J_{7}({J_{7}}^{2}-J_{30})-(J_{3}J_{7}+J_{4}J_{6})J_{20}+J_{6}J_{7}J_{14}-({J_{7}}^{2}-2J_{30})J_{13}.
		$

		\item[$S_{J_{15}}$ :]
		
		$2J_{7}J_{15}=-J_{1}^{2}J_{7}(J_{5}+J_{4})+J_{1}(J_{4}(J_{13}+J_{12})-J_{5}J_{13}-2J_{7}J_{9})+(J_{2}J_{4}+2J_{3}^{2})J_{7}+2J_{3}J_{24}-2J_{4}J_{21}+2J_{9}J_{13}$.

		\item[$S_{J_{16}}$ :]

		$2J_{7}J_{16}=J_{1}^{2}(J_{26}-J_{4}J_{7})
		+(2J_{9}-J_{1}J_{5})(3J_{13}-J_{12})+J_{1}[2J_{3}J_{14}+J_{4}(J_{13}+J_{12})
		-2J_{6}(J_{11}+2J_{10})-2J_{7}J_{9}]
		+J_{2}(J_{7}(J_{5}+J_{4})-2J_{26})+2(J_{3}(J_{3}J_{7}+J_{22}-(J_{24}+J_{23}))-J_{4}J_{21}+2J_{6}(J_{18}+J_{17})-J_{8}J_{14})$.

		\item[$S_{J_{17}}$ :]

		$4J_{7}J_{17}=J_{5}(J_{1}J_{14}+J_{5}J_{6}-J_{3}J_{7})+2J_{1}J_{7}J_{10}+2J_{3}(J_{26}+J_{25})+J_{4}(J_{6}(J_{5}-J_{4})-2(J_{24}+J_{23}-J_{22}))-2J_{6}J_{19}-2J_{9}J_{14}+2(J_{11}+J_{10})(J_{13}-J_{12})$.

		\item[$S_{J_{18}}$ :]
		
		$2J_{7}J_{18}=(J_{3}J_{7}-J_{1}J_{14})(J_{5}-J_{4})
		-2J_{3}(J_{26}+2J_{25})-J_{5}J_{6}(J_{5}+J_{4})+2J_{4}J_{23}+2J_{6}J_{19}+2J_{9}J_{14}+2(J_{11}+J_{10})J_{12}$.

		\item[$S_{J_{19}}$ :]

		$2J_{7}^{2}J_{19}=J_{5}^{2}J_{7}^{2}+2J_{4}J_{14}^{2}+4J_{7}J_{14}(J_{11}+J_{10})+4J_{25}^{2}$.
		
		\item[$S_{J_{21}}$ :]
		
		$J_{7}J_{21}=J_{12}J_{13}-J_{6}J_{24}$.
		
		\item[$S_{J_{22}}$ :]
		$2J_{7}J_{22}=J_{4}(J_{1}J_{20}-J_{6}J_{7}-2J_{29})+2(J_{3}(J_{30}-J_{7}^{2})-J_{13}J_{14})$.
		
		\item[$S_{J_{23}}$ :]
		
		$2J_{7}J_{23}=J_{1}(J_{5}J_{20}-2J_{7}J_{14})+2(J_{3}(J_{30}-J_{31})+J_{6}(J_{26}+J_{25})-J_{9}J_{20}+J_{12}J_{14})$.
		
		\item[$S_{J_{24}}$ :]
		
		$2J_{7}J_{24}=J_{1}(J_{4}J_{20}+2J_{7}J_{14})+2(J_{3}(J_{30}-J_{7}^{2})-J_{4}(J_{6}J_{7}+J_{29})+J_{6}J_{25}-2J_{13}J_{14})$. 
		
		\item[$S_{J_{26}}$ :]
		
		$J_{7}^{2}J_{26}=(2J_{25}-J_{4}J_{7})(J_{31}-J_{30})+2J_{14}J_{41}+J_{20}J_{35}
		-J_{30}(J_{26}+2J_{25})$.
		
		\item[$S_{J_{27}}$ :]
		
		$	
		4J_{7}J_{27}=4J_{4}J_{32}-8J_{6}J_{28}+2J_{7}J_{8}(J_{4}+J_{5})
		+2J_{2}(J_{7}J_{11}+J_{4}J_{14} -2J_{35}) +J_{3}(8(J_{34}-J_{4}J_{12}+J_{7}J_{9})-12J_{33}+4J_{6}(2J_{10}+J_{11})+J_{13}(3J_{4}-J_{5})-4J_{3}J_{14})	+J_{1}(J_{6}(2J_{19}-J_{4}J_{5}-J_{5}^{2})-4J_{7}(J_{17}+J_{18})-2J_{4}(J_{22}+J_{23})
		+2J_{3}(J_{26}-J_{7}(J_{4}+2J_{5})))+2J_{1}^{2}(J_{35}-J_{4}J_{14}+J_{7}(2J_{10}+J_{11}))$.

		\item[$S_{J_{28}}$ :]
		
		$12J_{7}J_{28}=J_{3}(4J_{7}J_{10}+6J_{35})
		+J_{4}(2J_{6}J_{10}-3J_{1}J_{26}-3J_{5}J_{13}+6J_{34})
		+J_{13}(6J_{19}-3J_{5}^{2})+J_{25}(12J_{9}-6J_{1}J_{5})+12J_{24}(J_{11}+J_{10})$

		\item[$S_{J_{29}}$ :]
		$J_{7}J_{29}=J_{20}(J_{1}J_{7}-J_{13})-J_{6}J_{30}$.

		\item[$S_{J_{31}}$ :]
		$J_{7}^{2}J_{31}=(J_{7}^{2}-J_{30})^{2}-2J_{7}J_{14}J_{20}+J_{4}J_{20}^{2}+J_{30}^{2}$.

		\item[$S_{J_{32}}$ :]
		$J_{7}J_{32}=J_{1}(J_{7}J_{23}+J_{12}J_{14})+J_{3}J_{40}-J_{8}J_{30}-J_{13}J_{23}+J_{12}J_{22}-J_{14}J_{21}$.

		\item[$S_{J_{33}}$ :]
		
		$2J_{7}J_{33}=J_{7}(J_{1}(J_{26}+2J_{25})+J_{4}(J_{13}-J_{12}))+J_{14}(2J_{24}-2J_{3}J_{7}-J_{4}J_{6})-2J_{25}(J_{13}+J_{12})$.
		
		\item[$S_{J_{34}}$ :]
		$J_{7}J_{34}=J_{6}J_{35}+J_{13}J_{26}+2J_{25}(J_{13}-J_{12})+2J_{14}J_{24}$.

		\item[$S_{J_{35}}$ :]
		$3J_{7}J_{35}=3(J_{14}J_{25}+(J_{7}^{2}-J_{30})(J_{11}+J_{10})-J_{4}J_{41})-J_{7}^{2}J_{10}$.

		\item[$S_{J_{36}}$ :]

		$24J_{7}J_{36}=24J_{18}J_{21}+12J_{9}(2J_{32}+J_{7}J_{8})+12J_{6}(J_{4}(J_{16}-J_{15})-J_{8}(J_{11}+J_{10}))+12J_{3}(2(J_{42}+J_{7}J_{16}+J_{6}(J_{18}+2J_{17}))+J_{21}(2J_{4}-J_{5})+J_{8}J_{14})+24J_{3}^{2}(J_{22}-2J_{24}-J_{3}J_{7})+6J_{2}(J_{6}(J_{5}^{2}-2J_{19})+J_{3}(2(J_{26}+4J_{25})-J_{7}(J_{5}+3J_{4})-2J_{7}J_{18})+6J_{1}(24J_{12}J_{17}-12J_{9}(J_{23}+J_{22})+12J_{8}J_{25}+6J_{6}(8J_{28}-3J_{4}J_{9})-12J_{5}(J_{32}+J_{7}J_{8})+J_{3}(48J_{33}-36J_{34}-48J_{7}J_{9}-8J_{6}(3J_{11}+8J_{10})+6J_{5}(2J_{13}+J_{12})+12J_{4}(2J_{12}-3J_{13})+24J_{3}J_{14})+12J_{2}(2J_{35}-J_{7}J_{11}-J_{4}J_{14}))+J_{1}^{2}(6J_{5}(J_{23}+J_{22})-12(J_{10}J_{12}+J_{7}(J_{18}+J_{17}))+6J_{4}(J_{23}-J_{24}+J_{22})+3J_{4}J_{6}(5J_{5}-J_{4})-24J_{3}(J_{26}+2J_{25})+6J_{3}J_{7}(5J_{5}+J_{4}))+12J_{1}^{3}(J_{7}(J_{11}+J_{10})-J_{35}+J_{4}J_{14})$.

		\item[$S_{J_{37}}$ :]
		
		$12J_{7}J_{37}=3J_{4}[2(2J_{42}-J_{7}J_{16}-J_{6}J_{18}-J_{21}(3J_{5}+J_{4}))+J_{2}(J_{7}(J_{4}+2J_{5}) -2(J_{26}+J_{25}))+J_{3}(2(J_{22}-J_{24}-2J_{23}+J_{3}J_{7})+J_{6}(5J_{5}-2J_{4}) )] +J_{1}[J_{5}(6(J_{34}-J_{33})+9(J_{7}J_{9}+J_{1}J_{25}-J_{3}J_{14})-5J_{6}(3J_{11}+J_{10})-21/2J_{1}J_{5}J_{7})+3J_{4}(3J_{5}(J_{13}+J_{12})+J_{1}(2J_{26}+J_{25}-7/2J_{5}J_{7})
		+J_{4}(J_{13}+J_{12}-J_{1}J_{7})+J_{6}(J_{11}+3J_{10})-2(J_{34}+J_{33})+J_{7}J_{9}-J_{3}J_{14}) +J_{7}(6J_{1}J_{19}-4J_{3}J_{10} -12J_{28})+6J_{3}J_{35}] + J_{3} [12J_{7}J_{17}-24J_{43}+3J_{6}(6J_{19}-J_{5}^{2})+6J_{5}(5J_{24}-3J_{22})+4J_{10}(J_{13}-J_{12})+6J_{3}(J_{26}+2J_{25}+J_{5}J_{7}) ]  +12J_{9}(J_{33}-J_{34}) +3J_{5}(J_{2}(J_{5}J_{7} -6J_{25})+2(J_{6}J_{18}-J_{7}J_{16}))-4J_{8}(3J_{35}+J_{7}J_{10})-8J_{6}(3J_{38}+J_{9}J_{10})
		$.

		\item[$S_{J_{38}}$ :]

		$24J_{7}J_{38}=12J_{9}J_{35}+6J_{24}(J_{5}^{2}-2J_{19})
		+12J_{7}(J_{5}(2J_{18}+3J_{17})+J_{3}J_{19})+J_{4}[24J_{43}+12J_{7}(J_{1}J_{10}-J_{17})+3J_{6}(J_{5}^{2}+J_{4}J_{5}-2J_{19})+6J_{5}J_{24}-6J_{3}(J_{26}+2J_{25}-2J_{5}J_{7})+4J_{10}(J_{12}-J_{13})]
		-6J_{1}J_{5}(J_{35}+J_{7}(4J_{11}+6J_{10})) $.

		\item[$S_{J_{39}}$ :]

		$24J_{7}J_{39}=9J_{5}\left[
		J_{7}(2J_{19}-(J_{5}+J_{4})J_{5})+2J_{25}(J_{4}-J_{5})\right]
		+12J_{4}(6J_{44}+J_{10}J_{14})+\\
		+36\left[ J_{19}J_{25}+J_{35}(J_{10}+J_{11})\right]-8J_{7}J_{10}^{2}$.

		\item[$S_{J_{40}}$ :]
		
		$J_{7}J_{40}=J_{12}J_{30}+J_{20}J_{24}$.

		\item[$S_{J_{41}}$ :] 
		$J_{7}J_{41}=J_{20}J_{25}+J_{14}(J_{30}-J_{7}^{2})$.

		\item[$S_{J_{42}}$ :]

		$2J_{7}J_{42}=2J_{3}(J_{5}J_{29}-2J_{45})-J_{6}(J_{1}J_{35}+J_{6}J_{19}-4J_{7}J_{17}
		+16J_{43})+J_{24}(4J_{24}+6J_{23})+1/2(J_{5}+J_{4})(J_{5}J_{6}^{2}-12J_{12}J_{13}+8J_{7}J_{21}) +2J_{4}(J_{6}J_{23}-J_{8}J_{20}-2J_{12}^{2}+6J_{12}J_{13}-3J_{7}J_{21}
		)+6J_{12}(J_{34}-J_{33})-2/3J_{6}J_{10}(J_{13}+J_{12})-4J_{21}(J_{26}+J_{25})+2J_{22}(2J_{3}J_{7}-J_{24}-4J_{23})+(9J_{3}J_{6}-2J_{2}J_{7})(J_{26}+2J_{25})+(4(J_{3}^{2}-J_{15})+2J_{2}J_{4})(J_{30}-J_{31})
		+4J_{7}J_{9}J_{12}+2J_{13}J_{34}-4J_{14}J_{32}-4J_{18}J_{29}$.

		\item[$S_{J_{43}}$ :]
		
		$12J_{7}J_{43}=(6J_{35}+4J_{7}J_{10})(J_{13}-J_{12})-3(2J_{24}+J_{4}J_{6})(J_{26}+2J_{25})
		$.

		\item[$S_{J_{44}}$ :]
		$6J_{7}J_{44}=3J_{25}(J_{5}J_{7}-J_{26})-J_{14}(2J_{7}J_{10}+3J_{35})-6J_{25}^{2}$.

		\item[$S_{J_{45}}$ :]
		
		$2J_{7}J_{45}=2J_{7}^{2}(J_{22}-J_{24})+J_{4}(J_{7}J_{29}-J_{12}J_{20})+2(J_{12}(J_{7}J_{14}+J_{41})+J_{24}J_{31})$.

		\item[$S_{J_{46}}$ :]
		
		$J_{7}J_{46}=J_{14}(J_{4}J_{20}-2(J_{7}J_{14}+J_{41}))-J_{20}J_{35}+J_{26}J_{30}
		+2J_{25}(J_{30}-J_{31})$.
		
		\item[$S_{J_{47}}$ :]

		$2J_{7}J_{47}=J_{20}(J_{4}J_{20}-2(J_{41}+J_{7}J_{14}))+2J_{30}J_{31}$.

		\item[$S_{K_{2}}$ :]	
		$J_{7}K_{2}=J_{6}K_{3}+J_{13}K_{1}$.
		
		\item[$S_{K_{4}}$ :]
		
		$J_{7}K_{4}=K_{1}(K_{7}-J_{1}K_{3}-J_{3}K_{1})+K_{2}K_{3}$.

		\item[$S_{K_{5}}$ :]
		
		$2J_{7}K_{5}=2K_{3}^{2}+J_{4}K_{1}^{2}$.

		\item[$S_{K_{6}}$ :]
		$J_{7}K_{6}=J_{12}K_{3}+J_{24}K_{1}$.
		
		\item[$S_{K_{7}}$ :]	
		$2J_{7}K_{7}=2J_{13}K_{3}-J_{4}J_{6}K_{1}$.
		
		\item[$S_{K_{8}}$ :]
		$J_{7}K_{8}=J_{25}K_{1}+J_{14}K_{3}$.

		\item[$S_{K_{9}}$ :]
		
		$J_{7}K_{9}=K_{1}K_{17}+K_{3}K_{6}$.

		\item[$S_{K_{10}}$ :]
		
		$2J_{7}K_{10}=2J_{13}K_{5}+K_{1}(J_{4}(J_{1}K_{1}-2K_{2})-2J_{3}K_{3})$.

		\item[$S_{K_{11}}$ :]
		
		$J_{7}K_{11}=K_{1}K_{18}+K_{3}K_{8}$.

		\item[$S_{K_{12}}$ :]
		$J_{7}K_{12}=J_{20}K_{3}-J_{30}K_{1}$. 
		
		\item[$S_{K_{13}}$ :]
		
		$J_{7}K_{13}=K_{1}(J_{7}K_{3}+J_{14}K_{1}-K_{29})+K_{3}K_{12}$.

		\item[$S_{K_{14}}$ :]
		
		$J_{7}K_{14}=J_{24}K_{2}+J_{21}K_{3}-J_{6}K_{17}$.

		\item[$S_{K_{15}}$ :]
		
		$2J_{7}K_{15}=K_{1}(J_{4}(J_{13}-J_{12})-2J_{3}J_{14})+2(J_{14}K_{7}-J_{13}K_{8}+J_{22}K_{3})$.
		
		\item[$S_{K_{16}}$ :]

		$2J_{7}K_{16}=K_{1}(J_{4}(J_{13}-J_{12})-2(J_{3}J_{14}+J_{33}K_{1}))+2(K_{2}(J_{26}K_{2}+J_{25})-J_{6}(K_{19}+K_{18}))+4(J_{14}K_{6}-J_{12}K_{8})+2J_{23}K_{3}$.

		\item[$S_{K_{17}}$ :]
		
		$2J_{7}K_{17}=J_{1}(2J_{7}K_{8}+J_{4}(K_{12}+2J_{7}K_{1}))+2J_{3}(K_{29}-J_{14}K_{1}-J_{7}K_{3})-4J_{13}K_{8}+2J_{6}K_{18}-J_{4}(2K_{28}+J_{7}K_{2}+J_{6}K_{3})$.

		\item[$S_{K_{18}}$ :]
		
		$2J_{7}K_{18}=2J_{25}K_{3}-K_{1}(2J_{7}(J_{11}+J_{10})+J_{4}J_{14})$. 
		
		\item[$S_{K_{19}}$ :]
		
		$3J_{7}K_{19}=K_{1}(J_{7}(3J_{11}+2J_{10})+3(J_{4}J_{14}-J_{35}))+3J_{26}K_{3}$.

		\item[$S_{K_{20}}$ :]
		
		$J_{7}K_{20}=K_{1}(K_{1}K_{8}+K_{3}^{2}-K_{32})+K_{3}K_{13}$.

		\item[$S_{K_{21}}$ :]

		$J_{7}K_{21}=K_{1}(J_{1}K_{17}-J_{8}K_{3})+J_{24}K_{4}-K_{2}K_{17}+K_{3}K_{14}$.

		\item[$S_{K_{22}}$ :]
		
		$2J_{7}K_{22}=J_{1}J_{4}(K_{13}+K_{1}K_{3})-J_{4}(2K_{31}+J_{6}K_{5})+2J_{3}(K_{32}-K_{1}K_{8}-J_{7}K_{5})-2J_{13}K_{11}$.

		\item[$S_{K_{23}}$ :]
		$J_{7}K_{23}=2(K_{3}K_{16}+K_{6}K_{8})-(J_{12}K_{11}+J_{14}K_{9}+J_{23}K_{5})$.

		\item[$S_{K_{24}}$ :]
		$2J_{7}K_{24}=2K_{3}K_{17}-J_{4}K_{1}K_{6}$.

		\item[$S_{K_{25}}$ :]
		$J_{7}K_{25}=J_{25}K_{5}-K_{1}(K_{3}(J_{11}+J_{10})+J_{4}K_{8})$.
		
		$2J_{7}K_{25}=2K_{3}K_{18}-J_{4}K_{1}K_{8}$. 
		
		\item[$S_{K_{26}}$ :]

		$3J_{7}K_{26}=3K_{5}(J_{26}+4J_{25})-K_{1}(K_{3}(6J_{11}+8J_{10})+6K_{38})-12K_{3}K_{18}$.
		
		\item[$S_{K_{27}}$ :]
		$J_{7}K_{27}=K_{1}^{2}K_{11}+2K_{3}K_{20}-K_{5}K_{13}$.

		\item[$S_{K_{28}}$ :]
		$J_{7}K_{28}=K_{1}(J_{3}J_{20}+J_{7}J_{12})+J_{13}K_{12}-J_{20}K_{7}+J_{29}K_{3}$.

		\item[$S_{K_{29}}$ :]
		$2J_{7}K_{29}=2J_{30}K_{3}+J_{4}J_{20}K_{1}$.

		\item[$S_{K_{30}}$ :]
		$2J_{7}K_{30}=2J_{31}K_{3}+J_{20}(J_{4}K_{1}-2K_{8})+2J_{14}K_{12}$.

		\item[$S_{K_{31}}$ :]
		$J_{7}K_{31}=K_{1}^{2}(J_{24}-J_{22})-J_{29}K_{5}+2K_{3}K_{28}$.

		\item[$S_{K_{32}}$ :]
		
		$2J_{7}K_{32}=2J_{30}K_{5}+K_{1}(J_{4}(2K_{12}+J_{7}K_{1})-2J_{14}K_{3})$.
		
		\item[$S_{K_{33}}$ :]
		
		$2J_{7}K_{33}=K_{1}^{2}(2J_{26}+J_{4}J_{7})+2K_{1}(J_{7}K_{8}-2J_{14}K_{3}+J_{4}K_{12})+2(J_{31}K_{5}-J_{20}K_{11}+J_{14}K_{13})$.

		\item[$S_{K_{34}}$ :]
		
		$J_{7}K_{34}=(J_{3}K_{1}-K_{7})(J_{24}+J_{23})+J_{13}(K_{17}+K_{16})-J_{12}K_{17}+J_{24}K_{6}+J_{32}K_{3}$.

		\item[$S_{K_{35}}$ :]
		
		$2J_{7}K_{35}=J_{4}[2K_{48}-J_{23}K_{1}-J_{14}K_{2}-J_{12}K_{3}+J_{7}(K_{7}-2K_{6}-J_{3}K_{1}+J_{1}K_{3})+J_{1}(J_{14}K_{1}-K_{29})]+J_{3}(2K_{49}+J_{26}K_{1}-2J_{14}K_{3})+2J_{13}K_{19}$.

		\item[$S_{K_{36}}$ :]
		$6J_{7}K_{36}=K_{1}(12J_{43}+3J_{5}J_{24}-3J_{4}J_{23}+J_{10}(3J_{12}-4J_{13}))-6K_{8}(2J_{24}+J_{23})-6K_{7}(J_{26}+2J_{25})+6(J_{14}K_{16}-J_{12}K_{18}+3J_{25}K_{6}+2J_{34}K_{3}-J_{35}K_{2})$.

		\item[$S_{K_{37}}$ :]

		$J_{7}K_{37}=K_{1}^{2}(K_{17}-K_{15})+2K_{3}K_{31}-K_{5}K_{28}$

		\item[$S_{K_{38}}$ :]

		$4J_{7}K_{38}=K_{1}(J_{7}(2J_{19}-J_{4}J_{5}-J_{5}^{2})+2J_{4}J_{26})
		+4(K_{3}(J_{35}-J_{4}J_{14})+J_{4}J_{7}K_{8})$

		\item[$S_{K_{39}}$ :]
		
		$2J_{7}K_{39}=K_{3}(J_{4}K_{1}^{2}-2K_{1}K_{8}+4K_{32})+2K_{5}(J_{14}K_{1}-K_{29})$

		\item[$S_{K_{40}}$ :]
		
		$6J_{7}K_{40}=3K_{3}(2K_{33}+K_{1}^{2}(J_{4}-J_{5}))-K_{1}(6(K_{51}+K_{1}K_{19}+J_{7}K_{11})-3J_{4}K_{13}+2J_{10}K_{1}^{2})$.

		\item[$S_{K_{41}}$ :]
		$J_{7}K_{41}=J_{12}K_{22}+J_{22}K_{9}-J_{32}K_{5}+2K_{3}K_{34}-2K_{6}K_{15}$.

		\item[$S_{K_{42}}$ :]
		
		$2J_{7}K_{42}=J_{12}(K_{26}+4K_{25})+K_{3}(4K_{35}-2J_{18}K_{1})
		+K_{9}(J_{26}+2J_{25})-2K_{6}(K_{19}+3K_{18})+2(J_{24}K_{11}-J_{28}K_{1}^{2}-J_{33}K_{5}+K_{8}K_{17}-2J_{14}K_{24})$.

		\item[$S_{K_{43}}$ :]
		
		$J_{7}K_{43}=J_{12}(K_{26}+2K_{25})+2(K_{8}K_{17}-J_{14}K_{24}+K_{3}K_{36}
		-K_{6}(K_{19}+K_{18}))+J_{26}K_{9}-J_{34}K_{5}$.

		\item[$S_{K_{44}}$ :]
		$J_{7}K_{44}=K_{2}(K_{3}K_{5}-K_{39})-K_{1}(K_{1}K_{22}+K_{5}K_{6})+K_{3}(K_{37}-K_{3}K_{4})+K_{4}K_{32}$.
		\item[$S_{K_{45}}$ :]
		
		$J_{7}K_{45}=J_{14}K_{26}+J_{26}K_{11}-J_{35}K_{5}+2K_{3}K_{38}-2K_{8}K_{19}$.
		
		\item[$S_{K_{46}}$ :]

		$2J_{7}K_{46}=2K_{3}(K_{39}-K_{1}K_{11})+J_{4}K_{1}(K_{20}+K_{1}K_{5})$.

		\item[$S_{K_{47}}$ :]

		$2J_{7}K_{47}=K_{1}^{2}(J_{4}K_{5}-2(K_{26}+K_{25}))+K_{1}(J_{4}K_{20}-2(K_{56}+K_{3}K_{11}))+2K_{3}K_{40}$.
		
		\item[$S_{K_{48}}$ :]
		
		$J_{7}K_{48}=J_{12}((K_{29}-K_{30})-J_{14}K_{1})+J_{20}(K_{17}+K_{16})-J_{23}K_{12}+J_{31}K_{6}$.

		\item[$S_{K_{49}}$ :]
		$J_{7}K_{49}=K_{1}(J_{7}J_{25}-J_{14}^{2})+J_{14}K_{29}-J_{30}K_{8}+J_{41}K_{3}$. 
		
		\item[$S_{K_{50}}$ :]
		
		$J_{7}K_{50}=K_{1}(J_{24}K_{3}-J_{25}K_{2})+J_{12}K_{32}+J_{14}K_{31}-J_{41}K_{4}+K_{2}K_{49}+K_{3}K_{48}-K_{6}K_{29}-K_{8}K_{28}$.

		\item[$S_{K_{51}}$ :]

		$6J_{7}K_{51}=6(J_{41}K_{5}-K_{1}(J_{14}K_{8}+K_{12}(J_{10}+J_{11})))
		+J_{7}(3J_{14}K_{5}-6K_{1}K_{19}-K_{1}^{2}(5J_{10}+3J_{11})-3J_{7}K_{11})$.

		\item[$S_{K_{52}}$ :]
		
		$2J_{7}K_{52}=J_{5}(J_{7}K_{14}-J_{21}K_{3})+K_{1}(J_{8}J_{26}-J_{4}(J_{7}J_{8}+J_{32})+2J_{21}(J_{11}+J_{10}))+2[J_{14}K_{34}+J_{21}(K_{19}+K_{18})-K_{14}(J_{26}+J_{25})-J_{32}K_{8}+J_{42}K_{3}]$.

		\item[$S_{K_{53}}$ :]
		
		$12J_{7}K_{53}=K_{1}(3J_{4}(J_{34}-J_{5}J_{13})+3J_{13}(2J_{19}-J_{5}^{2})+J_{24}(4J_{10}+6J_{11})-12J_{7}J_{28})+6(J_{24}K_{19}-J_{26}K_{17}+2J_{43}K_{3})$.

		\item[$S_{K_{54}}$ :]

		$2J_{7}K_{54}=K_{1}^{2}(K_{3}(J_{1}(J_{5}-J_{4})-2J_{9})+2J_{4}K_{6})
		+2(J_{7}K_{1}K_{24}-J_{12}(J_{24}K_{20}K_{39}+K_{1}K_{11}))$.

		\item[$S_{K_{55}}$ :]
		
		$4J_{7}K_{55}=2J_{26}(K_{1}(J_{11}+J_{10})-K_{18})
		+2J_{25}(K_{1}(3J_{11}+2J_{10})+K_{19})-J_{4}J_{5}J_{14}K_{1}+4(J_{14}K_{38}-J_{35}K_{8}+J_{44}K_{3})$.

		\item[$S_{K_{56}}$ :] 
		
		$2J_{7}K_{56}=2K_{1}^{2}(K_{3}(J_{11}+J_{10})+J_{4}K_{8})+K_{1}(J_{7}(K_{26}-J_{5}K_{5}+6K_{25})-2J_{14}K_{11})+2(J_{25}K_{20}J_{14}K_{39}-J_{7}K_{3}K_{11})$

		\item[$S_{K_{57}}$ :]

		$2J_{7}K_{57}=J_{5}(J_{7}K_{21}-K_{3}K_{14})+J_{21}(K_{26}+2K_{25})+K_{34}(2K_{8}-J_{4}K_{1})-J_{26}K_{21}+2(K_{3}K_{52}-J_{32}K_{11}-K_{14}K_{18})$

		\item[$S_{K_{58}}$ :]

		$6J_{7}K_{58}=2J_{10}(K_{2}K_{8}-J_{14}K_{4})-K_{1}(6(K_{17}(J_{11}+2J_{10})+J_{18}K_{8}+J_{28}K_{3}) +K_{15}(6J_{11}+4J_{10}))+3[K_{26}(J_{24}+J_{23}+J_{22})+J_{26}(K_{23}-K_{22}+K_{24}(J_{26}+4J_{25})]
		+6[K_{18}(K_{16}-K_{15}-2K_{17}) -K_{19}(K_{17}+K_{16}) +K_{8}(2K_{36}+K_{35})+K_{38}(K_{6}-2K_{7}) +J_{35}(K_{10}-K_{9})-K_{11}(J_{34}+J_{33})+K_{25}(J_{22}-J_{23})+J_{13}K_{45}-J_{14}K_{43}
		+J_{43}K_{5}]
		$.

		\item[$S_{K_{59}}$ :]

		$J_{7}K_{59}=K_{1}(K_{5}K_{17}-K_{6}K_{11})+K_{6}K_{39}+K_{17}K_{20}$.

		\item[$S_{K_{60}}$ :]
		$6J_{7}K_{60}=3K_{18}(J_{5}K_{3}+K_{1}(3J_{11}+4J_{10})+3K_{19}+6K_{18})-2K_{8}(9K_{38}+J_{10}K_{3})+3J_{14}K_{45}+6(K_{3}K_{55}-J_{26}K_{25}-J_{25}(K_{26}+4K_{25})+2J_{35}K_{11}-J_{44}K_{5})$.

		\item[$S_{K_{61}}$ :]
		$J_{7}K_{61}=K_{11}(K_{3}^{2}-K_{32})+2K_{8}(K_{39}-K_{3}K_{5})+K_{13}K_{25}$.

		\item[$S_{K_{62}}$ :]
		
		$2J_{7}K_{62}=J_{5}(J_{7}K_{28}-J_{29}K_{3})-J_{12}(J_{26}K_{1}+2K_{49})+2(J_{22}(K_{29}-K_{30})+K_{15}(J_{31}-J_{30})+K_{1}(J_{14}J_{23}-J_{18}J_{20})+J_{20}K_{35}-J_{33}K_{12}+J_{41}K_{6}+J_{45}K_{3}).$
		
		\item[$S_{K_{63}}$ :]
		
		$J_{7}K_{63}=K_{30}(2J_{25}-J_{26})+K_{19}(J_{31}-2J_{30})-K_{1}(J_{7}^{2}J_{11}+J_{14}J_{26})+J_{46}K_{3}+2(J_{26}K_{29}-J_{31}K_{18})$.

		\item[$S_{K_{64}}$ :]
		
		$6J_{7}K_{64}=K_{1}^{2}(2J_{10}J_{12}+J_{7}(2J_{1}J_{10}-6J_{17}))
		+3K_{1}(K_{3}(J_{5}J_{12}-J_{4}J_{12})+J_{7}(K_{3}(2J_{9}+J_{1}(J_{4}-J_{5}))-2J_{4}K_{6})+2J_{12}K_{19})+6J_{7}^{2}(K_{22}-K_{24})+J_{12}(6(J_{7}K_{11}+K_{51})-3J_{4}K_{13})+3J_{4}J_{7}K_{31}+6J_{24}K_{33}$.

		\item[$S_{K_{65}}$ :]

		$2J_{7}K_{65}=K_{1}^{2}(J_{4}J_{26}-4J_{44})-2J_{46}K_{5}+4K_{3}K_{63}$

		\item[$S_{K_{66}}$ :]
		
		$2J_{7}K_{66}=2K_{1}[ J_{18}K_{13} +K_{9}(J_{26}+J_{25}) -K_{10}( J_{26}+2J_{25}) +K_{5}(J_{34} -J_{33}) -K_{8}(K_{16}+3K_{15})+K_{1}K_{53}-K_{3}K_{36}+ 2K_{6}K_{18}]+K_{28}(K_{26}+6K_{25}) +K_{20}(3J_{33}-J_{34})-K_{37}(3J_{26}+4J_{25})-K_{12}(K_{43}+2K_{42}) +2[K_{31}(K_{19}-K_{18})+K_{32}(3K_{15}-K_{16})+K_{39}(J_{23}-3J_{22})+K_{3}(2(J_{22}K_{5}-K_{3}K_{15})+K_{64})+2(K_{8}K_{50}-K_{11}K_{48})+K_{13}K_{36}]$.

		\item[$S_{K_{67}}$ :]

		$12J_{7}K_{67}=12K_{3}K_{65}+K_{1}(6J_{11}K_{33}+3J_{14}(J_{5}-2K_{26}-6K_{25})+2J_{10}(3K_{33}+K_{32}+K_{5}(2K_{1}K_{8}-J_{7})+J_{4}K_{1}^{2})+3J_{7}(K_{11}(2J_{4}-J_{5})-2K_{45})+6J_{11}K_{32}+3J_{4}(J_{5}K_{13}+2K_{1}K_{19}))$.

		\item[$S_{K_{68}}$ :]
		
		$4J_{7}K_{68}=4((2K_{17}-J_{9}K_{1})K_{40}-J_{24}K_{47}+J_{12}(K_{61}+K_{5}K_{11}))
		+2J_{7}(J_{3}(K_{46}-K_{5}^{2})-K_{5}K_{24}-K_{4}K_{25})-2J_{4}J_{12}K_{27}
		+J_{1}(J_{7}(J_{4}K_{27}-2K_{5}K_{11})+2(J_{5}-J_{4})K_{1}K_{40})$.
		
		\item[$S_{K_{69}}$ :]
		
		$J_{7}K_{69}=J_{26}K_{46}-J_{35}K_{27}+2(K_{3}K_{67}-K_{1}(K_{3}K_{45}-K_{5}K_{38})-K_{19}K_{39}+K_{20}K_{38})-K_{5}K_{65}-K_{13}K_{45}+K_{26}K_{32}$.

		\item[$S_{K_{70}}$ :]
		
		$2J_{7}K_{70}=J_{20}K_{1}(J_{5}J_{7}-2J_{26}-4J_{25})+2(K_{29}(J_{7}^{2}-2J_{31})
		+J_{30}(K_{30}-2J_{7}K_{3})+J_{47}K_{3})$.

		\item[$S_{K_{71}}$ :]

		$J_{7}K_{71}=K_{1}(J_{14}(2K_{29}-K_{30})
		-(J_{20}K_{18}+J_{41}K_{3}))+K_{3}K_{70}
		+K_{29}(J_{7}K_{3}+K_{30}-2K_{29})+J_{30}(2K_{32}-K_{1}K_{8}-K_{3}^{2}-K_{33})$.

		\item[$S_{K_{72}}$ :] 
		
		$2J_{7}K_{72}=2K_{1}^{2}(2J_{14}K_{8}+K_{12}(J_{11}+2J_{10})+J_{7}(K_{19}
		-K_{3}(J_{5}+J_{4}))+2J_{4}K_{30})+2J_{30}K_{40}-2(K_{56}+K_{1}K_{26})
		+J_{7}(K_{1}(2J_{4}K_{13}-6K_{51}-5J_{14}K_{5}+3J_{7}K_{11})+K_{1}^{3}(3J_{11}+J_{10})-2J_{14}K_{20})+J_{4}(J_{20}(K_{20}-K_{1}K_{5})+J_{14}K_{1}^{3})$.

		\item[$S_{K_{73}}$ :]

		$2J_{7}K_{73}=K_{1}(K_{8}(K_{1}K_{8}+(3K_{33}-2K_{32})+2K_{3}^{2})+K_{3}K_{51}+K_{5}K_{49}-3K_{11}K_{30}+K_{12}K_{26}-K_{13}(K_{19}+2K_{18}))+K_{3}^{4}-K_{3}^{2}(K_{33}+K_{32})
		+3K_{3}(K_{72}-K_{5}K_{29})+K_{5}(2J_{7}K_{32}+K_{71})+2(K_{12}K_{56}+K_{30}K_{39})
		+K_{29}(2K_{39}-K_{40})-3K_{32}K_{33}$.

		\item[$S_{K_{74}}$ :]
		$6J_{7}K_{74}=K_{1}^{3}(6J_{5}K_{11}+12K_{67}+K_{5}(6J_{11}-4J_{10})-15K_{45})
		+K_{1}^{2}(K_{20}(3J_{11}+5J_{10})+K_{39}(6J_{4}-9J_{5})-3J_{4}K_{3}K_{5})
		+3K_{1}(K_{27}(J_{4}J_{7}-J_{26})+J_{14}(3K_{46}-4K_{5}^{2}-K_{47})-3J_{7}K_{61})
		+3(K_{3}(2K_{73}-J_{14}K_{27})+J_{7}K_{8}K_{27})$.

		\item[$S_{K_{75}}$ :]
		
		$J_{7}K_{75}=K_{1}(3K_{11}(K_{39}-K_{40})
		-2K_{27}(K_{19}+K_{18})-K_{20}K_{26})+K_{3}(K_{11}K_{20}-2K_{8}K_{27})
		+K_{5}(J_{7}K_{46}-K_{5}K_{32}+2K_{8}K_{20}-K_{11}K_{13}+K_{73})+2(K_{32}K_{47}-K_{33}K_{46})$.

	\end{itemize}
	
\end{theorem}

\begin{proof}
	
	Since $J_{1},...,J_{47},K_{1},...,K_{75}$ still invariant under any centro-affine transformation, they still particularly invariant after the transformation $q$\ref{varch} . In each element of $\mathcal{S}(2,\mathbb{R},\left \{ 0,1,3\right \} )$ we substitute each tonsorial coefficient by its expression in $J_{1},...,J_{47},K_{1},...,K_{75}$ then lead on syzygies. For example,

	\begin{equation*}
	J_{41}=b^{ \alpha }b^{\beta }b^{ \gamma }b^{ \delta }b_{\delta \nu \tau }^{
		\mu }b_{\gamma \mu p}^{\nu }b_{\alpha \beta q}^{\tau }\varepsilon ^{pq}=\frac{1}{J_{7}}a^{\alpha }a^{\beta }a^{\gamma }a^{\delta }a_{\delta \nu \tau}^{\mu }a_{\gamma \mu p}^{\nu }a_{\alpha \beta q}^{\tau }\varepsilon ^{pq}
	\end{equation*}
	where $\alpha ,\beta ,\gamma,\delta ,\mu ,\nu ,\tau ,p,q=1,2$ i.e.
	
	\begin{itemize}
		\item[$J_{41}=$] $\displaystyle \sum_{\alpha =1}^{2}\sum_{\beta
			=1}^{2}\sum_{\gamma =1}^{2}\sum_{\delta =1}^{2}\sum_{\mu =1}^{2}\sum_{\eta
			=1}^{2}\sum_{\tau =1}^{2}\sum_{p=1}^{2}\sum_{q=1}^{2}b^{ \alpha }b^{\beta
		}b^{ \gamma }b^{ \delta }b_{\delta \nu \tau }^{ \mu }b_{\gamma \mu p}^{\nu
		}b_{\alpha \beta q}^{\tau }\varepsilon ^{pq}$
		
		\item[$=$] $-{(b^{{1}})}^{2}{(b^{{2}})}^{2}b^{2}_{111}{b^{2}_{222}}^{2} + {(b^{
				{1}})}^{3}b^{{2}}b^{1}_{122}{b^{1}_{111}}^{2}-{(b^{{1}})}^{4}b^{2}_{111}b^{2}_{112}b^{1}_{122}-{(b^{{1}})}^{4}b^{1}_{111}b^{1}_{122}b^{2}_{111}+\\
		+{(b^{{2}})}^{2}{(b^{{1}})}^{2}b^{1}_{222}{b^{1}_{111}}^{2}+{(b^{{1}})}^{4}{b^{1}_{112}}^{2}b^{2}_{111}-%
		{(b^{{1}})}^{4}b^{2}_{111}{b^{2}_{122}}^{2}-{(b^{{1}})}^{3}b^{{2}}{
			b^{1}_{112}}^{2}b^{2}_{112}-b^{{1}}{(b^{{2}})}^{3}b^{1}_{122}{b^{1}_{112}}
		^{2}\\-{(b^{{1}})}^{2}{(b^{{2}})}^{2}{b^{1}_{112}}^{3}+2\,{(b^{{1}})}^{2}{(b^{{
					2}})}^{2}{b^{1}_{122}}^{2}b^{2}_{111}+{(b^{{2}})}
		^{4}b^{1}_{112}b^{2}_{222}b^{1}_{122}-{(b^{{2}})}^{4}{b^{1}_{122}}
		^{2}b^{1}_{112}-{(b^{{1}})}^{4}b^{1}_{111}b^{2}_{122}b^{2}_{112}\\
		+{(b^{{1}})}
		^{4}{b^{2}_{112}}^{2}b^{2}_{122}-2\,{(b^{{2}})}^{2}{(b^{{1}})}^{2}b^{1}_{222}
		{b^{2}_{112}}^{2}-{(b^{{2}})}^{3}b^{{1}}b^{2}_{112}{b^{2}_{222}}^{2}+{(b^{{2}
			})}^{4} b^{2}_{222}b^{2}_{112}b^{1}_{222}+b^{{2}}{(b^{{1}})}^{3}b^{2}_{112}{
			b^{2}_{122}}^{2}\\+{(b^{{2}})}^{4}b^{1}_{222}b^{2}_{112}b^{1}_{122}+{(b^{{2}})}
		^{2}{(b^{{1}})}^{2}b^{1}_{112}{b^{2}_{122}}^{2}+{(b^{{2}})}^{3}b^{{1}
		}b^{1}_{122}{b^{2}_{122}}^{2}+{(b^{{2}})}^{2}{(b^{{1}})}^{2}{b^{2}_{122}}
		^{3}\\+{(b^{{2}})}^{2}{(b^{{1}})}^{2}b^{1}_{111}b^{2}_{222}b^{1}_{112}+{(b^{{2}
			})}^{2}{(b^{{1}})}^{2}b^{2}_{222}b^{2}_{111}b^{1}_{122}-2\,{(b^{{1}})}^{3}b^{
			{2}}b^{2}_{111}b^{2}_{122}b^{2}_{222}+{(b^{{1}})}^{3}b^{{2}
		}b^{1}_{111}b^{2}_{122}b^{1}_{112}\\
		+{(b^{{2}})}^{3}b^{{1}
		}b^{2}_{222}b^{2}_{111}b^{1}_{222}-{(b^{{1}})}^{3}b^{{2}
		}b^{1}_{111}b^{1}_{222}b^{2}_{111}-{(b^{{1}})}^{3}b^{{2}
		}b^{2}_{111}b^{2}_{112}b^{1}_{222}-2\,{(b^{{2}})}^{4}{b^{1}_{122}}
		^{2}b^{2}_{122}\\
		+{(b^{{2}})}^{4}b^{1}_{222}{b^{1}_{112}}^{2}+2\,{(b^{{1}})}
		^{4}b^{1}_{112}{b^{2}_{112}}^{2}-{(b^{{2}})}^{4} b^{1}_{222}{b^{2}_{122}}
		^{2}-3\,b^{{1}}{(b^{{2}})}^{3}b^{1}_{112}b^{2}_{122}b^{1}_{122}\\
		-2\,{(b^{{1}})
		}^{2}{(b^{{2}})}^{2}b^{1} _{112}b^{2}_{112}b^{1}_{122}+{(b^{{2}})}^{3}b^{{1}
		}b^{1}_{222}b^{2}_{111}b^{1}_{122}+{(b^{{2}})}^{3}b^{{1}}b^{1}_{122}
		b^{2}_{222}b^{2}_{112}+2\,{(b^{{1}})}^{2}{(b^{{2}})}
		^{2}b^{1}_{122}b^{2}_{122}b^{2}_{112}\\
		-{(b^{{2}})}^{3}b^{{1}}{b^{1}_{122}}
		^{2}b^{1}_{111}+{(b^{{1}})}^{3}b^{{2}}{b^{2}_{112}}^{2}b^{2}_{222}-{(b^{{1}})
		}^{3}b^{{2}}b^{1}_{111}b^{2}_{222}b^{2} _{112}-2\,b^{{1}}{(b^{{2}})}
		^{3}b^{2}_{122}b^{2}_{112}b^{1}_{222}\\
		-{(b^{{1}})}^{2}{(b^{{2}})}
		^{2}b^{1}_{111}b^{1}_{222}b^{2} _{112}-{(b^{{2}})}^{3}b^{{1}
		}b^{1}_{112}b^{2}_{222}b^{2}_{122}-{(b^{{1}})}^{3}b^{{2}}b^{1}_{111}{
			b^{2}_{122}}^{2}-{(b^{{1}})} ^{2}{(b^{{2}})}
		^{2}b^{1}_{111}b^{2}_{222}b^{2}_{122}\\
		+{(b^{{2}})}^{3}b^{{1}}{b^{2}_{122}}
		^{2}b^{2}_{222}-{(b^{{2}})}^{2} {(b^{{1}})}
		^{2}b^{1}_{112}b^{2}_{222}b^{2}_{112}-{(b^{{1}})}^{3}b^{{2}}{b^{1}_{112}}
		^{2}b^{1}_{111}+2 {(b^{{1}})}^{3}b^{{2}}b^{1}_{122}b^{2}_{111}b^{1}_{112}\\
		+3\,{(b^{{1}})}^{3}b^{{2}}b^{1}_{112}b^{2}_{122}b^{2}_{112} +{(b^{{2}})}^{3}b^{{1}}{
			b^{1}_{112}}^{2}b^{2}_{222} +2\,{(b^{{2}})}^{3}b^{{1}
		}b^{1}_{222}b^{1}_{111}b^{1}_{112}-{(b^{{1}})} ^{2}{(b^{{2}})}^{2}{
			b^{1}_{112}}^{2}b^{2}_{122}\\
		+{(b^{{2}})}^{3}b^{{1}}b^{1}_{111}b^{2}_{222}b^{1}_{122}+{(b^{{1}})}^{2} {(b^{{2}})}
		^{2}b^{1}_{111}b^{2}_{122}b^{1}_{122}-{(b^{{1}})}^{3}b^{{2}
		}b^{1}_{111}b^{2}_{112}b^{1}_{122}$
	\end{itemize}
	
	Since $b^{ 1}=0$ and $b^{2}=1$ then
	
	\begin{itemize}
		\item[$J_{41}=$] ${b^{1}_{112}}^{2}b^{1}_{222}-b^{1}_{112}{b^{1}_{122}}^{2}+b^{1}_{112}b^{1}_{122}b^{2}_{222}-2\,{b^{1}_{122}}^{2}b^{2}_{122}+b^{1}_{122}b^{1}_{222}b^{2}_{112}+b^{1}_{222}b^{2}_{112}b^{2}_{222}-b^{1}_{222}{b^{2}_{122}}^{2}$

	\item[$=$] $-{\frac {( 1/2\,J_{{4}}J_{{20}}-J_{{7}}J_{{14}}) ^{2}J_{{20}}}{{J_{{7}}}^{4}}}-{\frac {( 1/2\,J_{{4}}J_{{20}}-J_{{7}}J_{{14}} ) ({J_{{7}}}^{2}-J_{{30}} ) ^{2}}{{J_{{7}}}^{4}}}+{\frac { ( 1/2\,J_{{4}}J_{{20}}-J_{{7}}J_{{14}} ) ( {J_{{7}}}^{2}-J_{{30}} ) J_{{30}}}{{J_{{7}}}^{4}}}+{\frac {J_{{20}}( J_{{7}}J_{{14}}-1/2\,J_{{4}}J_{{20}} ) ^{2}}{{J_{{7}}}^{4}}}\newline-2{\frac { ({J_{{7}}}^{2}-J_{{30}}) ^{2}( J_{{7}}J_{{14}}-1/2\,J_{{4}}J_{{20}} ) }{{J_{{7}}}^{4}}}-{\frac {( {J_{{7}}}^{2}-J_{{30}}) J_{{20}}( 1/2\,J_{{4}} ( {J_{{7}}}^{2}-J_{{30}} ) -J_{{7}}J_{{25}})}{{J_{{7}}}^{4}}}-{\frac {J_{{20}}( 1/2\,J_{{4}} ( {J_{{7}}}^{2}-J_{{30}}) -J_{{7}}J_{{25}} ) J_{{30}}}{{J_{{7}}}^{4}}}$

		\item[$=$] ${\frac { \left( -{J_{{7}}}^{2}+J_{{30}} \right) J_{{14}}+J_{{20}}J_{{
						25}}}{J_{{7}}}}
		$
	\end{itemize}
or $J_{7}J_{41}= \left( -{J_{{7}}}^{2}+J_{{30}} \right) J_{{14}}+J_{{20}}J_{{25}}$ then lead to $
	S_{J_{41}}.$\newline
	Another example,
	
	\begin{equation*}
	K_{8}=a^{\alpha }a_{\delta \alpha p}^{\beta }a_{\gamma \beta q}^{\gamma
	}x^{\delta }\varepsilon ^{pq}
	\end{equation*}
	where $\alpha ,\beta ,\gamma,\delta ,p,q=1,2$ i.e.
	
	\begin{itemize}
		\item[$K_{8}=$] $
		-b^{1}b^{1}_{111}b^{1}_{122}x^{2}+b^{1}b^{1}_{111}b^{2}_{122}x^{1}+b^{1}{
			b^{1}_{112}}^{2}x^{2}-2
		\,b^{1}b^{1}_{112}b^{2}_{112}x^{1}+b^{1}b^{1}_{122}b^{2}_{111}x^{1}+b^{1}b^{2}_{222}b^{2}_{111}x^{1}\\
		-b^{1}b^{2}_{112}b^{2}_{122}x^{1}+b^{1}b^{2}_{222}b^{2}_{112}x^{2}-b^{1}{b^{2}_{122}}
		^{2}x^{2}-b^{2}b^{1}_{111}b^{1}_{122}x^{1}-b^{2}b^{1}_{111}b^{1}_{222}x^{2}+b^{2}
		{b^{1}_{112}}^{2}x^{1}\\
		+b^{2}b^{1}_{112}b^{1}_{122}x^{2}-b^{2}b^{1}_{112}b^{2}_{222}x^{2}+2\,b^{2}b^{1}_{122} b^{2}_{122}x^{2}-b^{2}b^{1}_{222}b^{2}_{112}x_{2}+b^{2}b^{2}_{222}b^{2}_{112}x^{1}-b^{2}
		{b^{2}_{122}}^{2}x^{1}$
		
		\item[$=$] $-b^{1}_{111}b^{1}_{122}x^{1}-b^{1}_{111}b^{1}_{222}x^{2}+{
			b^{1}_{112}}%
		^{2}x^{1}+b^{1}_{112}b^{1}_{122}x^{2}-b^{1}_{112}b^{2}_{222}x^{2}+2
		\,b^{1}_{122}b^{2}_{122}x^{2}-b^{1}_{222}b^{2}_{112}x_{2}\\
		+b^{2}_{222}b^{2}_{112}x^{1}-{b^{2}_{122}}^{2}x^{1}$

		\item[$=$] ${\frac {J_{{14}}K_{{3}}+J_{{25}}K_{{1}}}{J_{{7}}}}$\\

	\end{itemize}
or $J_{7}K_{8}=J_{{14}}K_{{3}}+J_{{25}}K_{{1}}$ then lead to $S_{K_{8}}.$\newline
	
	Similarly, we obtain $109$ syzygies between $J_{1},\ldots
	,J_{47},K_{1},\ldots ,K_{75}$.

	These syzygies are independent since each syzygy contain only one covariant of $\mathcal{S}(2,\mathbb{R},\left \{ 0,1,3\right \} )$ and generating. Indeed, a syzygy $S_{i}$ is of the form $\lambda _{i}J_{7}^{m_{i}}I_{j}+p_{i}$ where for $i=1,109$ $\lambda_{i}\in\mathbb{R}^{\ast },m_{i}\in\mathbb{N}^{\ast },I_{j}\in \mathcal{S}(2,\mathbb{R},\left \{ 0,1,3\right \} )$ and $p_{i}$ is the right hand side
	of $S_{i}.$ Then any syzygy $S$ between $J_{1},\ldots ,J_{47},K_{1},\ldots,K_{75}$ can be expressed in $p_{i}/a_{i}J_{7}^{m_{i}}$ then $S$ can be expressed in $p_{i},,i=1,109.$ Furthermore, these syzygies still hold at $%
	J_{7}=0$ by passing to the limit since $\mathcal{S}(2,\mathbb{R},\left \{ 0,1,3\right \} )$ is isomorphic
	to $\mathbb{R}^{6}$. The proof is completed.
\end{proof}

\end{document}